\documentclass[reqno,11pt]{amsart}

\usepackage{amsmath}
\usepackage{esint}
\usepackage{amsthm}
\usepackage{calligra}
\usepackage{epsfig}
\usepackage{psfrag}
\usepackage{graphicx}
\usepackage{graphpap,latexsym,epsf}
\usepackage{color}
\usepackage{amssymb,mathrsfs,enumerate}
\usepackage{endnotes}
\usepackage{calligra}
\usepackage{a4wide}
\usepackage[colorlinks, citecolor=citegreen, linkcolor=refred]{hyperref}
\usepackage{enumitem}
\usepackage{accents}
\usepackage[colorinlistoftodos]{todonotes}

\usepackage{mathtools}
\mathtoolsset{showonlyrefs}

\newcommand{\R}{\mathbb{R}}
\newcommand{\N}{\mathbb{N}}
\newcommand{\Z}{\mathbb{Z}}

\def\HHH{{\rm H}}
\def\RRR{{\mathrm R}}

\newcommand{\pa}{\partial}

\newcommand{\ep}{\varepsilon}

\newcommand{\Ric}{{\rm Ric}}

\newcommand{\Ro}{{\rm R}}
\newcommand{\D}{{\rm D}}

\newcommand{\na}{\nabla}

\newcommand{\vertl}{\vert\hspace{.07em}}
\newcommand{\vertr}{\hspace{.07em}\vert}
\newcommand{\xx}{\hspace{.07em}}
\def\ringg#1{\accentset{\circ}{#1}}
\newcommand{\Tau}{\mathrm{T}}

\mathchardef\emptyset="001F

\definecolor{citegreen}{rgb}{0,0.6,0}
\definecolor{refred}{rgb}{0.8,0,0}

\theoremstyle{plain}

\newtheorem{theorem}{Theorem}[section]

\newtheorem{lemma}[theorem]{Lemma}
\newtheorem{ackn}{Acknowledgement\ \!\!\!\!\!}

\theoremstyle{definition}
\newtheorem{definition}[theorem]{Definition}

\theoremstyle{remark}
\newtheorem{remark}[theorem]{Remark}

\numberwithin{equation}{section}

\makeatletter
\@namedef{subjclassname@2020}{\textup{2020} Mathematics Subject Classification}
\makeatother

\title[Riemannian Penrose inequality via Nonlinear Potential Theory]{Riemannian Penrose Inequality\\via Nonlinear Potential Theory}

\author[V.~Agostiniani]{Virginia Agostiniani}
\address{V.~Agostiniani, Universit\`a degli Studi di Trento, via Sommarive 14, 38123 Povo (TN), Italy}
\email{virginia.agostiniani@unitn.it}

\author[C.~Mantegazza]{Carlo Mantegazza}
\address{C.~Mantegazza, Universit\`a degli Studi di Napoli Federico II \& Scuola Superiore Meridionale, via Cintia (complesso universitario di Monte Sant'Angelo) 26, 80126 Napoli (NA), Italy}
\email{carlo.mantegazza@unina.it}

\author[L.~Mazzieri]{Lorenzo Mazzieri}
\address{L.~Mazzieri, Universit\`a degli Studi di Trento, via Sommarive 14, 38123 Povo (TN), Italy}
\email{lorenzo.mazzieri@unitn.it}

\author[F.~Oronzio]{Francesca Oronzio}
\address{F.~Oronzio,  Scuola Superiore Meridionale, Largo S. Marcellino, 10, 80138, Napoli, Italy}
\email{f.oronzio@ssmeridionale.it}

\date{\today}
\keywords{Monotonicity formulas, $p$--harmonic functions, ADM mass, geometric inequalities.}
\subjclass[2020]{53C21, 31C12, 31C15, 53Z05}

\begin{document}

\begin{abstract}
We provide a new proof of the Riemannian Penrose inequality for time--symmetric asymptotically flat initial data with a single black--hole horizon.
The proof proceeds through a newly established monotonicity formula holding along the level sets of the $p$--capacitary potential of the horizon boundary, in any asymptotically flat $3$--manifold with nonnegative scalar curvature. 
\end{abstract}

\maketitle

\section{A monotonicity formula for $p$--capacitary potentials.}

In the last decades, level set methods have proven to be a very powerful tools for the comprehension of geometric phenomena. A nowadays classical field of application is for example the existence theory for various  fundamental geometric evolution equations, notably, the {\em mean curvature flow}~\cite{CGG,ES1,ES2}, or the {\em inverse mean curvature flow}~\cite{HI}. On one hand, these theories are tailored for providing very general and flexible notions of weak solutions, on the other hand, understanding properly their qualitative behaviors and their geometric features usually requires a considerable amount of work. This is due not only to the inborn lack of regularity, but also to the fact that, in general, the evolution described by the weak solutions can be geometrically different from the one dictated by the classical theory, even for small times. Viceversa, classical solutions may count on a definitely less flexible existence theory, which is compensated by a much more transparent geometric behavior, as most of the formal computations that one can perform turn out to be rigorously justified.

A major challenge is then providing weak solutions with a sufficiently large tool--set for their qualitative analysis, eventually leading to relevant geometric conclusions. A significant and successful example in this sense is given by the masterful work of Huisken and Ilmanen in the proof of the {\em Riemannian Penrose inequality}~\cite{HI}. Briefly speaking, this inequality says that the total ADM mass and the area of the horizon boundary of an asymptotically flat $3$--manifold $(M,g)$ are related as follows,
\begin{equation}
m_{\mathrm{ADM}}\,\geq \,\sqrt{	\frac{|\pa M|}{16 \pi}}\,.
\end{equation}
In their celebrated paper, Huisken and Ilmanen settle the theory of weak solutions to the inverse mean curvature flow and -- even more remarkably -- provide them with an effective {\em monotonicity formula}, which in turn paves the way to the proof of the above geometric inequality. To put this fundamental work in perspective, it is worth recalling that the smooth counterpart of the whole procedure, namely, the {\em Geroch monotonicity formula} for the Hawking mass along the smooth inverse mean curvature flow~\cite{Ger}, was already available since the early seventies (see also~\cite{JW77}). In other words, it took more than twenty years to pass from the classical computation along the smooth flow to the full justification of the monotonicity formula along the weak flow. This is reflected in the fact that the proof of Huisken and Ilmanen is extremely sophisticated and requires a number of deep conceptual and technical insights. 

The aim of our work is to propose an alternative and simplified approach to the Riemannian Penrose inequality, where the weak inverse mean curvature flow is replaced by the level set flow of the $p$--capacitary potential of the black--hole horizon. This is a $p$--harmonic function solving the Dirichlet boundary value problem~\eqref{f1}, for which both the existence and the regularity theory are nowadays well understood. On this regard, it is also worth mentioning that, thanks to the result of Moser~\cite{Moser2007}, and its recent extension~\cite{MRS}, it is actually possible to use $p$--harmonic functions to recover the existence theory for the weak inverse mean curvature flow. However, at a first sight, in the absence of corresponding monotonicity formulas, the $p$--harmonic approach seemed to be essentially ineffective for drawing geometric conclusions. This perspective has changed completely after a series of recent works~\cite{Ago_Fog_Maz_2,Ben_Fog_Maz_1,FMP}, where some monotonicity formulas were established along the level sets flow of $p$--harmonic functions and subsequently employed to deduce new general versions of the Minkowski inequality in the classical Euclidean framework, as well as on complete manifolds with nonnegative Ricci curvature. 
These works were inspired and preceded by their harmonic counterparts~\cite{Ago_Fog_Maz_1,Ago_Maz_CV,Col_Min_2,Col_Min_3}, starting with Colding's breakthrough~\cite{Colding_Acta}. 

In a very recent paper~\cite{Ago_Maz_Oro_2}, a more sophisticated version of these monotonicity formulas was finally made available for the level sets flow of the Green's functions, in the context of asymptotically flat $3$--manifolds with nonnegative scalar curvature, leading to a simple proof of the {\em positive mass theorem} (see also~\cite{Ago_Maz_CMP,Ago_Maz_Oro_1,Bra_Kaz_Khu_Ste_2019,Hir_Kaz_Khu,Mun_Wan,Stern} for related results and methods). In the same paper~\cite[Section 3]{Ago_Maz_Oro_2} a Geroch--type computation was also performed along the smooth level sets flow of $p$--harmonic functions with nowhere vanishing gradient, to obtain a new proof of the Riemannian Penrose inequality under such (and some other minor) favorable assumption.  

In the present paper, we are going to analyze in full details the
general situation, where the $p$--capacitary potential of the horizon
boundary is allowed to have critical points, hence the corresponding
flow is possibly no longer smooth for all times, might experience 
jumps and the level sets could be subject to topological changes. 
In this spirit, our work
parallels the effort made by Huisken and Ilmanen in their extension of
the Geroch monotonicity result to the weak solutions of the inverse
mean curvature flow. On the other hand, we hope that lowering the
technical level of the proof would make the result more accessible, 
setting the stage for further investigations.

\begin{ackn} The authors are grateful to Riccardo~Benedetti,
  Stefano~Borghini and Bruno~Martelli for several useful discussions
  about the topological aspects of the proof. They also thank
  Luca~Benatti for some very useful insights about the asymptotic
  analysis involved in the final part of the proof. The authors are
  members of the Gruppo Nazionale per l'Analisi Matematica, la
  Probabilit\`a e le loro Applicazioni (GNAMPA), which is part of the
  Istituto Nazionale di Alta Matematica (INdAM).
\end{ackn}

\subsection{Setting and preliminaries.}

Let $(M,g)$ be a $3$--dimensional, complete, noncompact Riemannian manifold with nonnegative scalar curvature and smooth, compact and connected boundary $\partial M$. For $1<p<3$, let us {\em assume} that the following problem
\begin{equation}\label{f1}
\begin{cases}
\Delta_{p} u=0\ &\mathrm{in} \ {M}\\
\quad \,\,u=0 &\mathrm{on} \ \partial M\\
\quad \,\,u \to 1&\mathrm{\,at} \ \infty
\end{cases}
\end{equation}
admits a weak solution $u_p \in \mathscr{C}^{1, \beta} (M)\cap W^{1,p}(M)$, where $\Delta_pu={\mathrm{div}\,}\bigl( |\na u|^{p-2} \na u \bigr)$ denotes the $p$--Laplacian operator of $(M,g)$. Natural conditions ensuring the existence of the function $u_p$ will be introduced in Section~\ref{sec:RPI} (see Definition~\ref{def:AF}). By the maximum principle for $p$--harmonic functions (see~\cite[Lemma~3.18 and Theorem~6.5]{HKO}), such solution is unique and takes values in $[0,1)$. In particular, we have that $\partial M$ coincides with the level set $\{u_p=0\}$ and that $u_{p}:M\to [0,1)$ is proper, by virtue of the third condition in  problem~\eqref{f1}. We also observe that, by the results in~\cite{DiBenedetto1,DiBenedetto2}, the function $u_p$ is smooth outside its critical set
\begin{equation} 
\mathrm{Crit}(u_p) \, = \, \{ x \in M \, : \, \na u_p (x) = 0 \, \} \, .
\end{equation}
In particular, the boundary datum is attained smoothly, as the Hopf lemma (see~\cite[Section~2]{Ben_Fog_Maz_1} and references therein) implies that $t=0$ is a regular value of $u_p$.
 To proceed, we recall that for any $1<p<3$, the $p$--capacity of $\partial M$ is defined as 
\begin{equation}\label{pcapacity}
\mathrm{Cap}_{p}(\partial M) \,= \,\inf \bigg\{ \int\limits_{M}\vert \nabla v \vert ^{p}\,d\mu:\,v\in \mathscr{C}^{\infty}_{c}(M), \,v=1\,\,\text{on}\,\,\partial M\bigg\}
\end{equation}
and it is related to $u_p$ through the following identities
\begin{equation}\label{eq24}
\mathrm{Cap}_{p}(\partial M)\,=\int\limits_{M}\vert \nabla u_p\vert^{p}\,d\mu \,= \!\!\!\!\int\limits_{\{u_p=t\}} \!\!\!\vert \nabla u_p\vert^{p-1}\,d\sigma\,,
\end{equation}
for every regular value $t$ of $u_{p}$, see~\cite[Section~2]{Ben_Fog_Maz_1}. For the sake of notation, let us set
\begin{equation}\label{eqcp}
c_{p}\,=\,\bigg(\frac{\mathrm{Cap}_{p}(\partial M)}{4\pi}\bigg)^{\!\frac{1}{p-1}}\,.
\end{equation}
Let us also point out that, whenever there is no possibility of misunderstanding, we will systematically drop the subscript $p$ and we will simply denote by $u$ is the solution of  problem~\eqref{f1}. With these notations at hand, we now consider the vector field
\begin{equation}\label{X_p} 
X = \frac{c_p^{\frac{p-1}{3-p}}}{\left[ \,\frac{3-p}{p-1}\, (1-u) \right]^{\frac{p-1}{3-p}}} 
\left\{\frac{|\na u|^{p-2} \na u}{c_p^{p-1}} + \frac{ \nabla |\nabla u| - \frac{\Delta u}{|\na u|}\nabla u}{\phantom{\Bigl[}\,\frac{3-p}{p-1}\, (1-u) \phantom{\Bigr]}^{\phantom{1}}} \,+ \,\frac{|\na u| \na u}{\left[ \,\frac{3-p}{p-1}\, (1-u) \right]^2} \right\}\,.
\end{equation}
Notice that $X$ is well defined and smooth away from the critical points of $u$. Denoting by $\langle\,,\rangle$ the scalar product given by the metric $g$, we introduce the function 
\begin{equation}\label{defFp}
{F}_{p}(t)\,\,\,= \!\!\!\!\!\int\limits_{\{u=\alpha_{p}(t)\}} \!\!\! \left\langle X , \frac{\na u}{|\na u|}\right\rangle\,d\sigma \, ,
\end{equation}
where 
\begin{equation}\label{eqtp}
\alpha_{p}(t)=1-\Bigl(\frac{t_p}{t}\Bigr)^{\!\frac{3-p}{p-1}} \, , \qquad \text{ with } \qquad t_p = \Big(\,\frac{p-1}{3-p}\,c_p\Big)^{\!\frac{p-1}{3-p}} \, ,
\end{equation}
and the variable $t$ ranges in $\left[\,t_p, + \infty \right)$. The function $F_p$ is then well defined whenever $\alpha_p(t)$ is a regular value of $u$. To make the definition of $F_p$ more explicit, we observe that, expanding the equation $\Delta_{p} u=0 $ away from the critical points, one gets
\begin{equation}\label{eqff1000}
\Delta u\,=\,(2-p)\,\frac{\big\langle \na |\na u|,\na u\rangle}{\vert\nabla u\vert}\,.
\end{equation}
Consequently, the mean curvature $\HHH$ of a regular level set $\{u=\tau\}$ computed with respect to the outward unit normal $\na u/|\na u|$ 
can be expressed as 
\begin{equation}\label{eqff1001}
\HHH \,=\,\frac{\Delta u}{|\na u|}-\frac{\big\langle\na\vert \na u\vert, \na u\big\rangle}{|\na u|^2}\,=\,- (p-1) \,\frac{\big\langle\na\vert \na u\vert, \na u\big\rangle}{|\na u|^2} \,.
\end{equation}
Taking into account the above expression together with~\eqref{eq24} and~\eqref{eqcp}, it is easy to check that $F_p$ can be expressed as
\begin{equation}\label{eqfp}
{F}_{p}(t)\,= \,\,4\pi t \,-\,\frac{t^{\frac{2}{p-1}}}{c_{p}} \!\!\!\!\int\limits_{\{u=\alpha_{p}(t)\}} \!\!\!\! \vert \nabla u\vert \,\mathrm{H}\,d\sigma \,+ \,\frac{t^{\frac{5-p}{p-1}}}{c_{p}^{2}} \!\!\!\int\limits_{\{u=\alpha_{p}(t)\}} \!\!\!\vert \nabla u\vert^{2} \,d\sigma\,\,.
\end{equation}
We mention that setting $p=2$ and $c_p=1$ in this formula, one gets the same monotone quantity as the one employed in~\cite{Ago_Maz_Oro_2} to provide a Green's function proof of the {\em positive mass theorem}. 

\medskip

We are then ready to state the following monotonicity result. 

\begin{theorem}[Monotonicity along the regular values]\label{thm:monotonicity}
Let $(M,g)$ be a $3$--dimensional, complete, noncompact Riemannian manifold with nonnegative scalar curvature and smooth, compact and connected boundary $\partial M$. Moreover, assume that $H_2(M, \pa M;\Z)=\{0\}$ and suppose that, for every $1<p<3$, problem~\eqref{f1} admits a unique solution $u_p \in \mathscr{C}^{1, \beta} (M)\cap W^{1,p}(M)$. Let $s,t \in \left[\,t_p, + \infty \right)$ be such that $\alpha_p(s)$ and $\alpha_p (t)$ are regular values for $u_p$. Then, the following implication holds true
\begin{equation}
s \leq t \quad \Rightarrow \quad F_p(s) \leq F_p (t) \,,
\end{equation}
where $F_p$ is the function defined in formula~\eqref{defFp}.
\end{theorem}

The proof of this theorem will occupy the rest of the section and, for the sake of exposition, we proceed with a series of steps of increasing difficulty and generality. In the first step (Subsection~\ref{sub:nocrit}), we will treat the case where $u$ has no critical points. As this is the {\em smooth flow} case, the proof here reduces to a Geroch--type computation, similar to the one outlined in~\cite[Section 3]{Ago_Maz_Oro_2}. The second step (Subsection~\ref{sub:sard}) deals with the case where $u$ has a negligible set of critical values. This step constitutes the core of our analysis, as it faces the major conceptual and geometric difficulties caused by the presence of critical points. To understand this, one should consider that -- unlike for harmonic functions -- no {\em a priori} bound is available for the Hausdorff dimension of the critical set of a $p$--harmonic function, when $n \geq 3$ (the only known result is about $p$--harmonic functions in the plane and it is due to Alessandrini~\cite{Aless} and Manfredi~\cite{Manfredi}). As a consequence, the critical points of $u$ might be even arranged in subsets of positive top--dimensional measure. From the point of view of the level sets flow, these clusters of critical points would correspond to jumps, similar to the ones experienced by the weak inverse mean curvature flow of Huisken and Ilmanen. As such, they represent the most serious geometrical obstacle to the extension of the monotonicity formula of Subsection~\ref{sub:nocrit} to the weak flow. Also notice that integrating $|\na u|$ on the critical set $\mathrm{Crit}(u) = \{ x \in M \,:\, \na u (x) = 0 \}$ and using the coarea formula, it is immediate to deduce that the set $\{ \tau \in [0,1) \, : \, \mathscr{H}^{n-1} (\mathrm{Crit}(u) \cap \{u = \tau \} ) > 0 \,\}$ must have zero measure. Hence, the assumption of negligible critical values is by no means sufficient to exclude jumps along the flow. If $p=2$, one can get rid of this assumption by using Sard theorem, since harmonic functions are smooth. On the other hand, when $p\neq 2$, the optimal regularity available for $p$--harmonic functions is only $\mathscr{C}^{1,\beta}$, which is clearly not enough to invoke Sard theorem. Albeit it is a common belief that critical values of $p$--harmonic functions should not have positive measure, this possibility is not excluded by any result in the literature, so far. We will overcome this technical issue in the third step (Subsection~\ref{sub:approx}), by using an approximation scheme originally introduced by Di~Benedetto~\cite{DiBenedetto1}.

Apart from the case where the $p$--capacitary potential $u$ has no critical points, so that the level sets are all immediately diffeomorphic to the boundary $\partial M=\{u=0\}$, in treating the other two steps, we will employ the triviality of the second relative homology group. Thanks to an argument similar to the one of~\cite[Lemma~4.2]{HI}, this will ensure that any level set of $u$ is connected, even in the presence of critical points. It is important to notice that the condition $H_2(M, \pa M ; \mathbb{Z}) = \{ 0\}$ is automatically satisfied under the natural assumption that $(M,g)$ is an asymptotically flat {\em exterior region}, as observed in~\cite[Lemma~4.1]{HI}.

\subsection{Monotonicity without critical points.}\label{sub:nocrit}

With the help of the Bochner formula, the twice contracted Gauss equation and formulas~\eqref{eqff1000},~\eqref{eqff1001}, the divergence of the vector field $X$ in formula~\eqref{X_p} can be expressed as 
\begin{align}
\mathrm{div}\,X = \frac{c_p^{\frac{p-1}{3-p}} |\na u | }{\left[ \,\frac{3-p}{p-1}\, (1-u) \right]^{\frac{p-1}{3-p} +1}} \,&\left\{\frac{|\na u|^{p-1}}{c_p^{p-1}} -\frac{\,{\rm{R}}^{\Sigma}}{2}+ \frac{\vert\,\nabla^{\Sigma}\vert \nabla u\vert\,\vert^{2}}{\vert \nabla u \vert^{2}}+\frac{\Ro}{2}+\frac{\,\vert \ringg{\mathrm{h}}\vert^{2}}{2}
\phantom{\frac{\vert \nabla u \vert}{\phantom{\Bigl[}\,\frac{3-p}{p-1}\, (1-u) \phantom{\Bigr]}^{\phantom{1}}}} \quad \right.\\
&\left. 
\qquad
\qquad\qquad\,
+\,\frac{5-p}{p-1}\,\left( \frac{\vert \nabla u \vert}{\phantom{\Bigl[}\,\frac{3-p}{p-1}\, (1-u) \phantom{\Bigr]}^{\phantom{1}}}\,-\,\frac{\mathrm{H}}{2} \right)^{\!\!\!2} \,\,\right\}\label{div(Xp)geom}
\end{align}
where $\RRR^{\Sigma}(x)$, $\na^\Sigma$ and $\ringg{\mathrm{h}}(x)$ represent the scalar curvature, the Levi--Civita connection and the trace--free second fundamental form of the regular level set $\Sigma=\{u = u(x) \}$ passing through the point $x\in M$. Also notice that in order to establish the above formula, one can employ the following Kato--type identity for $p$--harmonic functions,
\begin{equation*}
|\na\na u|^2 - \left[ 1 + \frac{(p-1)^2}{2} \right]|\na|\na u||^2 \,\,= \,\,|\na u|^2 \big\vert \ringg{\mathrm{h}}\,\big\vert^{2} + \left[ 1 - \frac{(p-1)^2}{2} \right] |\na^\Sigma|\na u||^2 \,.
\end{equation*}
Let us observe that if $u$ has no critical points, then all the values in the range of $u$ are regular, and the monotonicity can be easily deduced by means of the following computation, using the divergence theorem, the coarea formula and identity~\eqref{div(Xp)geom}. 
\begin{align*}
F_p(t)- F_p(s) =&\,\!\!\!\int\limits_{\{u=\alpha_{p}(t)\}} \!\!\! \left\langle X , \frac{\na u}{|\na u|}\right\rangle\,d\sigma \,\,- \!\!\!\!\!\!\int\limits_{\{u=\alpha_{p}(s)\}} \!\!\! \left\langle X , \frac{\na u}{|\na u|}\right\rangle\,d\sigma \,\\
=&\,\!\!\!\int\limits_{\{\alpha_p(s) < u < \alpha_{p}(t)\}} \!\!\! \!\!\! \mathrm{div}\,X \,d\mu \,=\!\!\!\int\limits_{(\alpha_p(s) , \alpha_{p}(t))} \!\!\!\!\!\!\,d\tau \int\limits_{\{ u= \tau \}}\!\!\! \frac{\mathrm{div}\,X }{|\na u|} \,d\sigma \,\\
=&\,\!\!\!\int\limits_{(\alpha_p(s) , \alpha_{p}(t))} \!\!\! \frac{c_p^{\frac{p-1}{3-p}}\,d\tau }{\left[ \,\frac{3-p}{p-1}\, (1-\tau) \right]^{\frac{p-1}{3-p} +1}} \int\limits_{\{ u= \tau \}}\!\!\! \biggl( \frac{|\na u|^{p-1}}{c_p^{p-1}} -\frac{\,\rm{R}^{\Sigma}}{2} \biggr)\,d\sigma\\
&\,+ \!\!\!\int\limits_{(\alpha_p(s) , \alpha_{p}(t))} \!\!\! \frac{c_p^{\frac{p-1}{3-p}}\,d\tau }{\left[ \,\frac{3-p}{p-1}\, (1-\tau) \right]^{\frac{p-1}{3-p}+1}} \int\limits_{\{ u= \tau \}}\!\!\! \biggl( \frac{\vert\,\nabla^{\Sigma}\vert \nabla u\vert\,\vert^{2}}{\vert \nabla u \vert^{2}}+\frac{\Ro}{2}+\frac{\,\vert \ringg{\mathrm{h}}\vert^{2}}{2} \biggr) 
\,d\sigma\\
&\,+ \!\!\!\int\limits_{(\alpha_p(s) , \alpha_{p}(t))} \!\!\! \frac{c_p^{\frac{p-1}{3-p}}\,d\tau }{\left[ \,\frac{3-p}{p-1}\, (1-\tau) \right]^{\frac{p-1}{3-p} +1}} \int\limits_{\{ u= \tau \}}\!\!\!
\frac{5-p}{p-1}\,\left( \frac{\vert \nabla u \vert}{\phantom{\Bigl[}\,\frac{3-p}{p-1}\, (1-u) \phantom{\Bigr]}^{\phantom{1}}}\,-\,\frac{\mathrm{H}}{2}\right)^{\!\!\!2}d\sigma\\
 \geq&\,\!\!\!\int\limits_{(\alpha_p(s) , \alpha_{p}(t))} \!\!\! \frac{c_p^{\frac{p-1}{3-p}} \,\left[ \,4 \pi - 2 \pi \chi\big( \{ u = \tau \} \big) \right]_{\phantom{1}} }{\left[ \,\frac{3-p}{p-1}\, (1-\tau) \right]^{\frac{p-1}{3-p}+1}} \,d\tau \,\,\geq \,\,0 \,,
\end{align*}
Observe that in the last passage, we used the identities~\eqref{eq24} and~\eqref{eqcp} in combination with the Gauss--Bonnet theorem, obtaining the Euler characteristic $\chi\big( \{ u = \tau \} \big)$ of the level set $\{ u = \tau \}$. The conclusion follows since, in the absence of critical points, all the level sets are diffeomorphic to the boundary $\partial M=\{u=0\}$. As the latter is a smooth connected closed surface, one has that $4\pi - 2\pi \chi(\{ u={\tau} \}) \geq 0$, for every $\tau \in [0,1)$.

\subsection{Monotonicity with negligible critical values.}\label{sub:sard}

Let us now consider the case where the solution $u$ of problem~\eqref{f1} is allowed to have a nonempty set of critical points ${\mathrm{Crit}(u)}$. Before dealing with the general case, let us first prove the monotonicity of $F_p$ under the favorable assumption that the set of the critical values of $u$ is a negligible set. As already observed, this case contains all the major conceptual and geometric difficulties caused by the presence of critical points. In particular, we are going to show how the monotonicity formula of Subsection~\ref{sub:nocrit} still holds despite the possible presence of jumps along the flow. To make the computations simpler, it is convenient to set
\begin{align}\label{Y_ep}
Y & = \frac{c_p^{\frac{p-1}{3-p}}}{\left[ \frac{3-p}{p-1}(1-u) \right]^{\frac{p-1}{3-p}}} 
\left\{ \frac{ \nabla |\nabla u| - \frac{\Delta u}{|\na u|}\nabla u}{\phantom{\bigl[}\,\frac{3-p}{p-1}\, (1-u) \phantom{\Bigr]}^{\phantom{1}}} \,+ \,\frac{|\na u| \na u}{\left[ \frac{3-p}{p-1}(1-u) \right]^2} \right\}\\
& = \frac{c_p^{\frac{p-1}{3-p}}}{\left[ \frac{3-p}{p-1}(1-u) \right]^{\frac{p-1}{3-p}}} 
\left\{ \frac{ \nabla^\top |\nabla u| + (p-1) \na^\perp |\na u|} 
{\phantom{\bigl[}\,\frac{3-p}{p-1}\, (1-u) \phantom{\Bigr]}^{\phantom{1}}} \,+ \,\frac{|\na u|^2 }{\left[ \frac{3-p}{p-1}(1-u) \right]^2} \,\frac{\na u}{|\na u|} \right\} \,, 
\end{align}
where $\na^\top |\na u|$ and $\na^\perp |\na u|$ denote the tangential and the normal component of $\na|\na u|$, respectively. Observe that $\na^\top = \na^\Sigma$, by the well known properties of the Levi--Civita connection. Noticing that the vector fields $X$ and $Y$ are {well defined where $\na u \neq 0$ and that they are related by the formula
\begin{equation*}
X \,\,= \,\,c_p^{-\frac{(2-p)(p-1)}{(3-p)}} \frac{ |\na u|^{p-2} \na u }{\left[ \frac{3-p}{p-1}(1-u) \right]^{\frac{p-1}{3-p}}} \,\,+ \,\,Y \,,
\end{equation*}
it is immediate to observe}
\begin{align}
\mathrm{div}\,Y= \frac{c_p^{\frac{p-1}{3-p}} |\na u | }{\left[ \frac{3-p}{p-1}(1-u) \right]^{\frac{p-1}{3-p}+1}} \,&\left\{ -\frac{\,{\rm{R}}^{\Sigma}}{2} + \frac{\vert\,\nabla^{\top}\vert \nabla u\vert\,\vert^{2}}{\vert \nabla u \vert^{2}}+\frac{\Ro}{2}+\frac{\,\vert \ringg{\mathrm{h}}\vert^{2}}{2}
\phantom{\frac{\vert \nabla u \vert}{\left[ \frac{3-p}{p-1}(1-u) \right]^{\phantom{1}}}} \quad \right.\\
& \left. \,\,\,\,\, + \frac{5-p}{p-1}\,\left( \frac{\vert \nabla u \vert}{\phantom{\Bigl[}\,\frac{3-p}{p-1}\, (1-u) \phantom{\Bigr]}^{\phantom{1}}}\,-\,\frac{\mathrm{H}}{2}\right)^{\!\!\!2} \,\,\right\}. \quad\label{div(Yp)geom}
\end{align}
We now consider a sequence of smooth, nondecreasing cut--off functions $\eta_k:[0,+\infty) \to [0,1]$, with the following structural properties
\begin{equation}
\eta_k (\tau) \equiv 0\quad \text{in $\Bigl[0 \,,\frac{1}{2k}\,\Bigr]$}\,,\qquad 
0\leq \eta_k'(\tau)\leq 2 k\quad \text{in $\Bigl[\,\frac{1}{2k}\,,\frac{3}{2k}\,\Bigr]$},
\qquad\eta_k(\tau) \equiv 1\quad\text{in $\Bigl[\,\frac{3}{2k} \,,+\infty\!\Bigr)$}\,,
\end{equation}
for every $k \in \mathbb{N}$. It is immediate to realize that the functions $\eta_k$ monotonically converge to the characteristic function of $(0,+\infty)$ in the pointwise sense.
Using these cut--off functions, we introduce the  vector fields
\begin{equation}
Y_{k}\,= \,\,\eta_{{k}}\!\left( \!\frac{ (p-1) \vert \nabla u\vert}{\left[ \frac{3-p}{p-1}(1-u) \right]^{\frac{1}{3-p}}} \!\right) \,Y \,,\,
\end{equation}
which are well defined and smooth on the whole manifold. It is readily checked that, for every $x \in M\setminus \mathrm{Crit}(u)$, the sequence of vectors $Y_{k}(x)$ converges by construction to $Y(x)$, as $k \to + \infty$. An easy computation yields
\begin{align*}
\mathrm{div}\,Y_{k}\, = \, &\,\,\eta_{{k}}\!\left( \!\frac{ (p-1)\vert \nabla u\vert}{\left[ \frac{3-p}{p-1}(1-u) \right]^{\frac{1}{3-p}}} \!\right) \,\mathrm{div}\,Y\\[-1em]
& \,+\, \eta'_{{k}}\!\left( \!\frac{ (p-1) \vert \nabla u\vert}{\left[ \frac{3-p}{p-1}(1-u) \right]^{\frac{1}{3-p}}} \!\right) 
c_p^{\frac{p-1}{3-p}} 
\frac{ (p-1) |\na^\top|\na u||^2 + \left| (p-1) \na^\perp |\na u| + \frac{|\na u| \na u}{\,\,\,\left[ \frac{3-p}{p-1}(1-u) 
\right]^{\phantom{1}}_{\phantom{1}}}\right|^2 }{\left[ \frac{3-p}{p-1}(1-u) \right]^{\frac{3}{3-p}}}\\
 \geq \, &\,\eta_{{k}}\!\left( \!\frac{ (p-1)\vert \nabla u\vert}{\left[ \frac{3-p}{p-1}(1-u) \right]^{\frac{1}{3-p}}} \!\right) \,\mathrm{div}\,Y \,.
\end{align*} 
Thus, setting
\begin{equation*}
X_{k} \,\,= \,\,c_p^{-\frac{(2-p)(p-1)}{(3-p)}} \frac{ |\na u|^{p-2} \na u }{\left[ \frac{3-p}{p-1}(1-u) \right]^{\frac{p-1}{3-p}}}\,\,+ \,\,Y_{k} \,,
\end{equation*} 
one immediately gets
\begin{equation*}
\mathrm{div}\,X_{k} \, \geq \,\frac{ c_p^{-\frac{(2-p)(p-1)}{(3-p)}} |\na u |^p }{\left[ \frac{3-p}{p-1}(1-u) \right]^{\frac{p-1}{3-p}+1}} \,
+ \eta_{{k}}\!\left( \!\frac{ (p-1)\vert \nabla u\vert}{\left[ \frac{3-p}{p-1}(1-u) \right]^{\frac{1}{3-p}}} \!\right) \,\mathrm{div}\,Y \,.
\end{equation*} 
Now, we let $s,t\in [t_p, + \infty)$ be as in the statement of Theorem~\ref{thm:monotonicity}. As for large enough $k\in\N$ the vector field $X_{k}$ coincides with $X$ on the boundary of $\{\alpha_p(s) < u <\alpha_p(t) \}$, the divergence theorem yields
\begin{align*}
F_p(t) - F_p(s) & = \!\!\!\int\limits_{\{u=\alpha_{p}(t)\}} \!\!\! \left\langle X_{k} , \frac{\na u}{|\na u|}\right\rangle d\sigma \,\,- \!\!\!\!\!\!\int\limits_{\{u=\alpha_{p}(s)\}} \!\!\! \left\langle  X_{k} , \frac{\na u}{|\na u|}\right\rangle d\sigma \,\\
& =\!\!\!\int\limits_{\{\alpha_p(s) < u < \alpha_{p}(t)\}} \!\!\! \!\!\! \mathrm{div}\,X_{k} \,d\mu \,\\
&\geq \!\!\int\limits_{\{\alpha_p(s) < u < \alpha_{p}(t)\}} \!\!\! \left\{\frac{ c_p^{-\frac{(2-p)(p-1)}{(3-p)}} |\na u |^p }{\left[ \frac{3-p}{p-1}(1-u) \right]^{\frac{p-1}{3-p}+1}} \,
+ \eta_{{k}}\!\left( \!\frac{ (p-1)\vert \nabla u\vert}{\left[ \frac{3-p}{p-1}(1-u) \right]^{\frac{1}{3-p}}} \!\right) \,\mathrm{div}\,Y \right\} \,d\mu\,. 
\end{align*}
We now claim that the rightmost hand side  converges to 
\begin{equation*}
\int\limits_{\{\alpha_p(s) < u < \alpha_{p}(t)\} } \!\! \left\{\frac{ c_p^{-\frac{(2-p)(p-1)}{(3-p)}} |\na u |^p }{\left[ \frac{3-p}{p-1}(1-u) \right]^{\frac{p-1}{3-p}+1}} \,+ \,\mathbb{I}_{M\setminus\mathrm{Crit}(u)} \,\mathrm{div}\,Y \right\} \,d\mu \,,
\end{equation*}
as $k \to + \infty$, where $\mathbb{I}_{M\setminus\mathrm{Crit}(u)}$ denotes the characteristic function of $M\setminus\mathrm{Crit}(u)$. Indeed, using the Gauss equation in combination with formulas~\eqref{eqff1000} and~\eqref{eqff1001}, we can rework formula~\eqref{div(Yp)geom} to obtain $\mathrm{div}\,Y=P+D$, with
\begin{align}
P&=\frac{c_{p}^{\frac{p-1}{3-p}} |\na u | }{\left[ \,\frac{3-p}{p-1}\, (1-u) \right]^{\frac{p-1}{3-p} +1}} \biggl[\,\frac{\vert\,\nabla^{\top}\vert \nabla u\vert\vert^{2}}{|\na u|^2}\, +\,\vert \ringg{\mathrm{h}}\vert^{2}\,+\frac{\mathrm{H}^{2}}{2}+\,(2-p)(p-1)\,\frac{\big\langle\na |\na u|,\na u\big\rangle^2}{\vert \nabla u\vert^{4}}\biggr]\\
&=\frac{c_{p}^{\frac{p-1}{3-p}} |\na u | }{\left[ \frac{3-p}{p-1}(1-u) \right]^{\frac{p-1}{3-p} +1}} \biggl[\,\frac{\vert\,\nabla^{\top}\vert \nabla u\vert\vert^{2}}{|\na u|^2}+\vert \ringg{\mathrm{h}}\vert^{2}+\frac{(p-1)(3-p)}{2}\,\frac{\big\langle\na |\na u|,\na u\big\rangle^2}{\vert \nabla u\vert^{4}}\biggr]\\
D&=\frac{c_{p}^{\frac{p-1}{3-p}} |\na u | }{\left[ \,\frac{3-p}{p-1}\, (1-u) \right]^{\frac{p-1}{3-p} +1}}\biggl[\frac{\Ric(\na u,\na u)}{|\na u|^2}+\frac{(p-1)(5-p)}{(3-p)^{2}}\,\biggl(\frac{\vert \nabla u\vert^{2}}{(1-u)^{2}}+(3-p)\frac{\big\langle\na |\na u|,\na u\big\rangle}{(1-u)|\na u|}\,\biggr)\biggr] \, .
\end{align}
As usual, the above identities are valid outside ${\rm Crit}(u)$. We now notice that the term $P$ is nonnegative, while $D$ is bounded on every compact subset of $M$. The claim follows by applying the dominate convergence theorem to the sequence $\eta_k \, D$ and the monotone convergence theorem to the sequence $\eta_k\, P$. 

To complete our argument, we are now going to use the assumption that $\mathcal{N} = u({\rm Crit}(u))$ is negligible. 
Using the coarea formula, the identity~\eqref{div(Yp)geom} and the Gauss--Bonnet theorem, we can further develop the above computation, getting
\begin{align}
F_p(t) - F_p(s) & \geq \!\!\!\int\limits_{(\alpha_p(s) , \alpha_{p}(t)) } \!\!\! d\tau \int\limits_{\{ u= \tau \}}\! \left\{\frac{ c_p^{-\frac{(2-p)(p-1)}{(3-p)}} |\na u |^{p-1} }{\left[ \frac{3-p}{p-1}(1-u) \right]^{\frac{p-1}{3-p}+1}} \,
+ \,\mathbb{I}_{M\setminus\mathrm{Crit}(u)} \,\frac{\mathrm{div}\,Y}{|\na u|} \right\} d\sigma\\
& = \!\!\!\int\limits_{(\alpha_p(s) , \alpha_{p}(t))\setminus \mathcal{N} } \!\!\! d\tau \int\limits_{\{ u= \tau \}}\! \left\{\frac{ c_p^{-\frac{(2-p)(p-1)}{(3-p)}} |\na u |^{p-1} }{\left[ \frac{3-p}{p-1}(1-u) \right]^{\frac{p-1}{3-p}+1}} \,
+ \,\frac{\mathrm{div}\,Y}{|\na u|} \right\}\, d\sigma\\
& = \!\!\!\int\limits_{(\alpha_p(s) , \alpha_{p}(t))\setminus \mathcal{N} } \!\!\! d\tau \int\limits_{\{ u= \tau \}}\!\!\! \,\frac{\mathrm{div}\,X }{|\na u|}\,d\sigma\\
& \geq \!\!\!\int\limits_{(\alpha_p(s) , \alpha_{p}(t)) \setminus \mathcal{N}} \frac{c_p^{\frac{p-1}{3-p}} d\tau }{\left[ \frac{3-p}{p-1}(1-\tau) \right]^{\frac{p-1}{3-p}+1}} \int\limits_{\{ u= \tau \}}\!\!\! \biggl(\frac{|\na u|^{p-1}}{c_p^{p-1}} -\frac{\,\rm{R}^{\Sigma}}{2} \biggr)\,d\sigma\\
& = \!\!\!\int\limits_{(\alpha_p(s) , \alpha_{p}(t))\setminus \mathcal{N}} \!\!\! \frac{c_p^{\frac{p-1}{3-p}} \,\left[ \,4 \pi - 2 \pi \chi\big( \{ u = \tau \} \big) \right]_{\phantom{1}} }{\left[ \frac{3-p}{p-1} (1-\tau) \right]^{\frac{p-1}{3-p}+1}} \,d\tau \,\,\geq \,\,0 \,.\label{eqcar999}
\end{align}
The last inequality holds because all the regular level sets of $u$ are closed and {\em connected} surfaces, so that $4\pi - 2\pi \chi(\{ u={\tau} \}) \geq 0$, for every $\tau \in [0,1)\setminus \mathcal{N}$. 

\medskip

For the reader's convenience, we now include an argument, inspired by the one presented in~\cite[Lemma~4.2]{HI}, showing that the assumption $H_2(M, \pa M;\Z)=\{0\}$ is sufficient to infer the connectedness of the regular level sets of $u$. 

We let $\tau \in (0,1)$ be a regular value of $u$, and we consider the regular level set $\{ u= \tau \}$. We first observe that the sub--level set $\{ u < \tau \}$ is necessarily connected. More precisely, we claim that each connected component of $\{ u < \tau \}$ has a nonempty intersection with $\pa M = \{u=0\}$. Indeed, by the properness of $u$, every connected component of $\{u<\tau\}$ must be bounded. Now, if some of these connected components would not intersect $\pa M$, then 
$u$ would achieve an interior minimum on it, which is forbidden by the strong maximum principle (see~\cite[Theorem~6.5]{HKO}). As $\pa M$ is connected, it follows that $\{u<\tau\}$ is also connected. 

Let now $S \subseteq \{ u = \tau \}$ be a connected component of $\{ u = \tau \}$. As $H_2(M, \pa M;\Z)=\{0\}$, we have that the surface $S$ is homologous to (an integer multiple of) $\partial M$. We claim that $S$ is disconnecting $M$ into two connected components $A$ and $B$, so that $M = A \sqcup S \sqcup B$. In fact, if $M \setminus S$  would have a unique connected component, then there would exist a simple closed loop $\gamma$ intersecting transversally $S$ at a single point, with $\gamma \cap \pa M = \emptyset$. This is forbidden by the {\em intersection theory} (see~\cite[Chapter~6,~Section~11]{Bredon}), as the parity of the intersection with a given loop should be (generically) the same, within a given homology class. Hence, the claim is proven, and we have the decomposition $M = A \sqcup S \sqcup B$. 

Without loss of generality, we may assume that $\pa M \subseteq A$. By the previous reasoning, this implies that the sub--level set $\{ u < \tau \}$ is also contained in $A$, since both $A$ and $\{ u < \tau \}$ are open and connected and they both contain $\pa M$. In particular, we have that $\{ u = \tau \} = \pa \{u < \tau\} \subseteq A \sqcup S$, where the first equality holds because $\tau$ is a regular value of $u$. Now suppose by contradiction that $\{u = \tau \}$ contains another connected component $\Sigma$, with $\Sigma \cap S = \emptyset$. By virtue of the latter condition, $\Sigma$ lies necessarily in the interior of $A$. Hence, at any $x \in \Sigma \subseteq A$, we have that 
$$
u (x) \, = \,\tau \, = \, \max_{\pa A} u \, . 
$$
In other words, the point $x$ is an interior maximum for $u$ in $A$. Once again, this forbidden by the strong maximum principle (see~\cite[Theorem~6.5]{HKO}). Then, the only possibility is that $S$ is the unique connected component of $\{ u = \tau \}$.

\subsection{An approximate monotonicity formula.}\label{sub:approx}

To achieve the general result, we are going to locally approximate the $p$--capacitary potential $u_p$, solving problem~\eqref{f1}, with a family of smooth functions, solving a perturbed version of such problem. More concretely, we fix $T\in(0,1)$ such that $\{u_p= T \}$ is a regular level set of $u_p$,  and for every $\ep>0$, we consider the (unique) solution $u_{p}^{\ep}$ to the following problem
\begin{equation}\label{f2}
\begin{cases}
{\mathrm{div}\,}\bigl( |\na u|_\ep^{p-2} \na u \bigr)=0 &\mathrm{in} \ M_T={\{0\leq u_p\leq T\}} \, ,\\
\qquad \qquad\qquad \,\,\,\,\,u=0 &\mathrm{on} \ \partial M \, ,\\
\qquad \qquad\qquad\,\,\,\,\,u = T &\mathrm{on} \ \{u_p= T\} \, ,
\end{cases}
\end{equation}
where $|\na u|_\ep = \sqrt{|\na u|^2 + \ep^2}$. The functions $u_{p}^{\ep}$ were first introduced in~\cite{DiBenedetto1,DiBenedetto2} to establish the nowadays classical $\mathscr{C}^{1,\beta}$--regularity result for $p$--harmonic functions, as they actually converge in the $\mathscr{C}^{1,\beta}$--topology on the compact subsets of $M_T$ to the (a priori only) $W^{1,p}$--solution $u_p$ of problem~\eqref{f1}, when $\ep\to0$. In the same papers it is also proven that they converge {\em smoothly} to $u_p$ on the compact subsets of $M_T \setminus\mathrm{Crit}(u_p)$. To list some of the basic properties of the functions $u_p^\ep$, we observe that they satisfy a  nondegenerate quasilinear elliptic equation, with smooth coefficients, in divergence form. As such, the weak and the strong maximum principle as well as the Hopf lemma are in force, as it is proven in~\cite{serrin1970}. Using these fundamental tools, one can prove that the functions $u_p^\ep$ are smooth up to the boundary (see, e.g.~\cite{Lieb88}), that they take values in $[0,T]$ and that their gradient is never vanishing on the level sets $\{u_{p}^\ep=0\}=\partial M$ and $\{u_{p}^\ep=T\}=\{u_p=T\}$.

From our perspective, the key advantage of working with the functions $u_p^\ep$ instead of $u_p$ comes from the fact that they are smooth, hence, Sard theorem applies and the set of critical values is negligible. Adapting the  procedure described in the previous subsection, we eventually establish the validity of an {approximate monotonicity formula} in Lemma~\ref{thm:approx}. Theorem~\ref{thm:monotonicity}, will then be achieved in Subsection~\ref{sub:proof}, letting $\ep \to 0$.

To keep the notations simpler, we drop the subscripts $p$ and $\ep$, whenever there is no possibility of confusion and we simply write $u$ for $u_{p}^{\ep}$. In analogy with formula~\eqref{X_p}, {we define the following vector fields
\begin{equation}\label{X_pep}
X_\ep = \frac{c_{p,\ep}^{\frac{p-1}{3-p}}}{\left[ \,\frac{3-p}{p-1}\, (1-u) \right]^{\frac{p-1}{3-p}}} 
\left\{ \frac{|\na u|_\ep^{p-2} \na u}{c_{p, \ep}^{p-1}} + \frac{ \nabla |\nabla u| - \frac{\Delta u}{|\na u|}\nabla u}{\phantom{\Bigl[}\,\frac{3-p}{p-1}\, (1-u) \phantom{\Bigr]}^{\phantom{2}}} \,+ \,\frac{|\na u| \na u}{\left[ \,\frac{3-p}{p-1}\, (1-u) \right]^2} \right\} 
\end{equation}
where $\na u_{p}^{\ep}\neq 0$,} with the constant $c_{p,\ep}$ given by
\begin{equation}\label{eqcpep}
c_{p,\ep}^{p-1}\,=\,\frac{1}{4\pi} \int\limits_{\pa M} |\na u|_\ep^{p-2} |\na u| \,\,d \sigma 
\,= \,\frac{1}{4\pi} \int\limits_{\{u=\alpha^{\varepsilon}_{p}(t)\}} |\na u|_\ep^{p-2} |\na u| \,\,d \sigma \,.
\end{equation}
Notice that the second equality holds for every $t$ such that $\alpha_p^\ep(t)$ is a regular value of $u$, by divergence theorem. In analogy with formula~\eqref{defFp}, we introduce the functions
\begin{equation}\label{defFpep}
F_{p}^{\ep}(t)\,= \!\!\!\int\limits_{\{u=\alpha_{p}^\ep(t)\}} \!\!\! \left\langle X_\ep , \frac{\na u}{|\na u|}\right\rangle d\sigma \,,
\end{equation}
where 
\begin{equation}
\alpha_{p}^\ep(t)\,=\,1-\left(\frac{t_p^\ep}{t}\right)^{\!\!\frac{3-p}{p-1}} ,\qquad \hbox{with} \qquad t_p^\ep = \Big(c_{p,\ep}\,\,\frac{p-1}{3-p}\,\Big)^{\!\frac{p-1}{3-p}},
\end{equation}
and the variable $t$ ranges in $\big[\,t_p^\ep, t_p^\ep(1-T)^{\frac{1-p}{3-p}}\big]$. The assignment $F_p^\ep(t)$ is clearly well defined, whenever $\alpha_p^\ep(t)$ is a regular value of $u$. In order to make the expression of $F_p^\ep$ more explicit, we observe that, away from the critical set, the equation ${\mathrm{div}\,}\bigl( |\na u|_\ep^{p-2} \na u \bigr)=0 $ is equivalent to
\begin{equation}\label{eqff1}
\Delta u\,=\,(2-p)\,\frac{|\na u|^2}{\,\,\vert \nabla u\vert^{2}_{\ep}}\, \frac{\big\langle\na |\na u|,\na u\big\rangle}{|\na u|} \, .
\end{equation}
As a consequence, the mean curvature $\HHH$ of a regular level set of $u$, computed with respect to the unit normal $\na u/|\na u|$, is given by
\begin{equation}\label{eqff1bis}
\HHH \,=\,\frac{\Delta u}{|\na u|}-\frac{\big\langle\na\vert \na u\vert, \na u\big\rangle}{|\na u|^2}\,\,=\,-\frac{(p-1)\vert \na u\vert^{2}+\ep^{2}}{\vert \nabla u\vert^{2} + \ep^2}\,\frac{\big\langle\na\vert \na u\vert, \na u\big\rangle}{|\na u|^2}\,.
\end{equation}
Using these identities, $F_p^\ep(t)$ can be written as
\begin{equation}\label{defFpep2}
F_{p}^{\ep}(t)\,= \,\,\,4\pi t \,-\,\frac{t^{\frac{2}{p-1}}}{c_{p,\varepsilon}} \!\!\int\limits_{\{u=\alpha^{\varepsilon}_{p}(t)\}} \!\!\!\! \vert \nabla u\vert \,\mathrm{H}\,d\sigma \,+ \,\frac{t^{\frac{5-p}{p-1}}}{c_{p,\varepsilon}^{2}} \!\int\limits_{\{u=\alpha^{\varepsilon}_{p}(t)\}} \!\!\!\vert \nabla u\vert^{2} \,d\sigma \,.
\end{equation}
The above expression should be compared with formula~\eqref{eqfp}.

We can now state our approximate monotonicity Lemma.

\begin{lemma}\label{thm:approx}
Let $(M,g)$ be a $3$--dimensional Riemannian manifold satisfying the assumptions of Theorem~\ref{thm:monotonicity} and let $u$ be the solution to problem~\eqref{f2}. Let $s$ and $t$ be real values such that  
$$
t_p^\ep \, < \, s \, \leq \, t \, <\, t_p^\ep(1-T)^{\frac{1-p}{3-p}} \ ,
$$
and such that $\alpha_p^\ep(s)$ and $\alpha_p^\ep (t)$ are regular values for $u$. Then, the following inequality holds
\begin{equation} 
F_{p}^{\ep}(t) - F_{p}^{\ep} (s) \,\geq \,- \,\,\ep \,\,\Big(\frac{p+1}{p-1}\Big)^{\!2} \!\!\!\! \!\!\!\!\!\!\!\! \int\limits_{\{\alpha_p^\ep(s) < u < \alpha_{p}^\ep(t) \}} \!\!\!\!\!\!\!\!\!\!\!\! \frac{\ep |\na u|}{2(p+1)|\na u|^2 + 3 \ep^2}\,\frac{ c_{p,\ep}^{\frac{p-1}{3-p}} \,|\na u|^2}{\left[ \,\frac{3-p}{p-1}\, (1-u) \right]^{\frac{p-1}{3-p} + 3}} \,d\mu\,,
\end{equation}
where $F_p^\ep$ is the function defined in formula~\eqref{defFpep}.
\end{lemma}

\begin{proof} A long but straightforward computation, performed in the same spirit as the one leading to formula~\eqref{div(Xp)geom}, provide us with the expression for the divergence of $X_\ep$ away from the critical set,
\begin{align*}\label{div(Xep)geom}
\mathrm{div}X_\ep= \frac{c_{p,\ep}^{\frac{p-1}{3-p}} |\na u | }{\left[ \,\frac{3-p}{p-1}\, (1-u)  \right]^{\frac{p-1}{3-p} +1}}  \,&\left\{\frac{|\na u|_\ep^{p-2} |\na u|}{c_{p,\ep}^{p-1}} -\frac{\,{\rm{R}}^{\Sigma}}{2} + \frac{\vert\,\nabla^{\Sigma}\vert \nabla u\vert\,\vert^{2}}{\vert \nabla u \vert^{2}}+\frac{\Ro}{2}+\frac{\, \vert \mathring{\mathrm{h}}\vert^{2}}{2}  
\phantom{\frac{\vert \nabla u \vert}{\left[ \,\frac{3-p}{p-1}\, (1-u)  \right]^{\phantom{1}}}} \quad \right.\\
& + \frac{5-p}{p-1}\,\left( \frac{\vert \nabla u \vert}{\phantom{\Bigl[}\,\frac{3-p}{p-1}\, (1-u) \phantom{\Bigr]}^{\phantom{1}}} + \frac{p-1}{2}\,\frac{|\na u|^{2}}{|\na u|^{2}_{\ep}}\,\frac{\bigl\langle\na |\na u|,\na u\bigr\rangle}{|\na u|^{2}}\right)^{\!\!2}\\
& +\ep^2\, \bigg(\frac{p+1}{p-1}\bigg)^{\!\!2}\left( \frac{ a_\ep  \, \vert \nabla u \vert}{\phantom{\Bigl[}\,\frac{3-p}{p-1}\, (1-u) \phantom{\Bigr]}^{\phantom{1}}} + \frac{b_\ep \,  (p-1)}{2}
\frac{|\na u|^{2}}{|\na u|^{2}_{\ep}}\,\frac{\bigl\langle\na |\na u|,\na u\bigr\rangle}{|\na u|^{2}}\right)^{\!\!2}\\
&\left. -\ep^2 \bigg(\frac{p+1}{p-1}\bigg)^{\!\!2}   \,\frac{1}{2(p+1)|\na u|^2 + 3 \ep^2} \,\frac{ \vert \nabla u \vert^2}{\left[ \,\frac{3-p}{p-1}\, (1-u)  \right]^{2}}  \,\, \right\}\,,
\end{align*}
where the coefficients $a_\ep$ and $b_\ep$ are defined as
\begin{align*}
a_\ep  \, = \, \sqrt{\frac{ 1}{2(p+1) |\na u|^2 + 3 \ep^2}} \qquad \hbox{and} \qquad
b_\ep \,  = \,  \frac{\sqrt{2(p+1)|\na u|^2 + 3 \ep^2}}{(p+1) |\na u|^2}\,.
\end{align*}
It is easy to realize that, if $|\na u| \neq 0$ everywhere, then the thesis follows from a simple integration by parts, as in Subsection~\ref{sub:nocrit}. To treat the general case, it is convenient to proceed as in Subsection~\ref{sub:sard}, defining the vector field $Y_\ep$ as
\begin{align*}
Y_\ep & = \frac{c_{p,\ep}^{\frac{p-1}{3-p}}}{\left[ \,\frac{3-p}{p-1}\, (1-u)  \right]^{\frac{p-1}{3-p}}}   
\left\{   \frac{ \nabla |\nabla u|  -  \frac{\Delta u}{|\na u|}\nabla u}{\phantom{\Bigl[}\,\frac{3-p}{p-1}\, (1-u) \phantom{\Bigr]}^{\phantom{1}}}  \, + \, \frac{|\na u| \na u}{\left[ \,\frac{3-p}{p-1}\, (1-u)  \right]^2}  \right\}\\
& = \frac{c_{p,\ep}^{\frac{p-1}{3-p}}}{\left[ \,\frac{3-p}{p-1}\, (1-u)  \right]^{\frac{p-1}{3-p}}}   
\left\{   \frac{ \nabla^\top |\nabla u|  + (p-1) \na^\perp |\na u|} 
{\phantom{\Bigl[}\,\frac{3-p}{p-1}\, (1-u) \phantom{\Bigr]}^{\phantom{1}}}  \, + \, \frac{|\na u|^2 }{\left[ \,\frac{3-p}{p-1}\, (1-u)  \right]^2}  \, \frac{\na u}{|\na u|} \right.
\\
& \phantom{= \frac{c_p^{\frac{p-1}{3-p}}}{\left[ \,\frac{3-p}{p-1}\, (1-u)  \right]^{\frac{p-1}{3-p}}}  } 
\left. 
\phantom{ = \, \frac{c_p^{\frac{p-1}{3-p}}}{\left[ \,\frac{3-p}{p-1}\, (1-u)  \right]^{\frac{p-1}{3-p}}}   } 
\quad\qquad+ \frac{(2-p) \, \ep^2}{|\na u|^2 + \ep^2}\frac{ \na^\perp |\na u|} 
{\phantom{\Bigl[}\,\frac{3-p}{p-1}\, (1-u) \phantom{\Bigr]}^{\phantom{1}}} 
\right\} \, .
\end{align*}
Then, the fields $Y_\ep$ and $X_\ep$ are related through the formula
\begin{equation*}
X_\ep \,\,= \,\,c_{p, \ep}^{-\frac{(2-p)(p-1)}{(3-p)}} \frac{ |\na u|_\ep^{p-2} \na u }{\left[ \,\frac{3-p}{p-1}\, (1-u) \right]^{\frac{p-1}{3-p}}} \,\,+ \,\,Y_\ep \,.
\end{equation*}
Of course, the above definition makes sense only outside the critical set of $u$, however, using the cut--off functions $\eta_k$ introduced in Subsection~\ref{sub:sard}, we may consider, for every $k \in \mathbb{N}$, the vector fields
\begin{equation}
Y_{\ep,k}\, = \, \, \eta_{{k}}\!\left( \!\frac{ (p-1) \vert \nabla u\vert}{\left[ \,\frac{3-p}{p-1}\, (1-u)  \right]^{\frac{1}{3-p}}} \!\right) \, Y_\ep \, \quad \hbox{and} \quad  X_{\ep,k} \, = \, c_{p, \ep}^{-\frac{(2-p)(p-1)}{(3-p)}} \frac{ |\na u|_\ep^{p-2} \na u }{\left[ \,\frac{3-p}{p-1}\, (1-u)  \right]^{\frac{p-1}{3-p}}}\,\, + \, \, Y_{\ep,k} \,.
\end{equation}
Notice that $X_{\ep,k}$ and $Y_{\ep,k}$ are well defined and smooth on the whole $M_{T}$, as 
they are vanishing in a neighborhood of $\rm{Crit}(u)$. Computing as in Subsection~\ref{sub:sard}, we get
\begin{align}
\mathrm{div}&\,Y_{\ep,k}\, = \, \eta_{{k}}\!\left( \!\frac{ (p-1)\vert \nabla u\vert}{\left[ \frac{3-p}{p-1}(1-u) \right]^{\frac{1}{3-p}}} \!\right) \,\mathrm{div}\,Y_\ep \\[-1em]
\,+ \,& \, \eta'_{{k}}\!\left( \!\frac{ (p-1) \vert \nabla u\vert}{\left[ \frac{3-p}{p-1}(1-u) \right]^{\frac{1}{3-p}}} \!\right) 
c_{p,\ep}^{\frac{p-1}{3-p}} \,\, 
\frac{ (p-1) |\na^\top|\na u||^2 + \left| (p-1) \na^\perp |\na u| + \frac{|\na u| \na u}{\phantom{\bigl[}\,\frac{3-p}{p-1}\, (1-u) \phantom{\bigr]}^{\phantom{1}}}\right|^2 }{\left[ \frac{3-p}{p-1}(1-u) \right]^{\frac{3}{3-p}}}\\
\,+ \,& \, \eta'_{{k}}\!\left( \!\frac{ (p-1) \vert \nabla u\vert}{\left[ \frac{3-p}{p-1}(1-u) \right]^{\frac{1}{3-p}}} \!\right) 
c_{p,\ep}^{\frac{p-1}{3-p}} 
\frac{{\frac{(2-p) \ep^2}{|\na u|^2 + \ep^2}}\, \left(  \! (p-1)|\na^\perp|\na u||^2 +  \,\frac{|\na u| \,\left\langle\na |\na u| \,, \na u \right\rangle }{\phantom{\bigl[}\,\frac{3-p}{p-1}\, (1-u) \phantom{\bigr]}^{\phantom{1}}} \!\right) }{\left[ \frac{3-p}{p-1}(1-u) \right]^{\frac{3}{3-p}}}\\
&\quad\,\,= \, \eta_{{k}}\!\left( \!\frac{ (p-1)\vert \nabla u\vert}{\left[ \frac{3-p}{p-1}(1-u) \right]^{\frac{1}{3-p}}} \!\right) \,\mathrm{div}\,Y_\ep \\
\,+ \,&\,\eta'_{{k}}\!\left( \!\frac{ (p-1) \vert \nabla u\vert}{\left[ \,\frac{3-p}{p-1}\, (1-u) \right]^{\frac{1}{3-p}}} \!\right) 
{c_{p,\ep}^{\frac{p-1}{3-p}} }\,\,
\frac{    (p-1) |\na^\top|\na u||^2+\frac{\vert \nabla u\vert^{4}}{\left[ \,\frac{3-p}{p-1}\, (1-u) \right]^{2} }  + \frac{(p-1)^2|\na u|^2+(p-1)\ep^2}{|\na u|^2 + \ep^2} \, |\na^\perp|\na u||^2
}{ \left[ \,\frac{3-p}{p-1}\, (1-u) \right]^{\frac{3}{3-p}} } 
\\
\,+ \,&\,\eta'_{{k}}\!\left( \!\frac{ (p-1) \vert \nabla u\vert}{\left[ \,\frac{3-p}{p-1}\, (1-u) \right]^{\frac{1}{3-p}}} \!\right) 
{c_{p,\ep}^{\frac{p-1}{3-p}} }\,\, \frac{\,\,2(p-1)|\na u|^2+p\ep^2}{|\na u|^2 + \ep^2}\,\frac{|\na u| \,\left\langle\na |\na u| \,,\na u \right\rangle}{{ \left[ \,\frac{3-p}{p-1}\, (1-u) \right]^{\frac{6-p}{3-p}}}} 
\\
&\quad\,\geq \,\,\eta_{{k}}\!\left( \!\frac{ (p-1)\vert \nabla u\vert}{\left[ \frac{3-p}{p-1}(1-u) \right]^{\frac{1}{3-p}}} \!\right) \,\mathrm{div}\,Y_\ep \\
\,- \,&\,\eta'_{{k}}\!\left( \!\frac{ (p-1) \vert \nabla u\vert}{\left[ \,\frac{3-p}{p-1}\, (1-u) \right]^{\frac{1}{3-p}}} \!\right) 
\left( \!\frac{ (p-1) \vert \nabla u\vert}{\left[ \,\frac{3-p}{p-1}\, (1-u)  \right]^{\frac{1}{3-p}}} \!\right)^{\!\!\!2}\!\frac{{c_{p,\ep}^{\frac{p-1}{3-p}}}}{(p-1)^2}\max_{M_T} \frac{\frac{2(p-1)|\na u|^2+p\ep^2}{|\na u|^2 +\ep^2}\, |\na\na u| }{{ \left[ \,\frac{3-p}{p-1}\, (1-u) \right]^{\frac{4-p}{3-p}} }} \,.
\end{align}
Noticing that 
\begin{equation}
\frac{2(p-1)|\na u|^2+p\ep^2}{|\na u|^2 +\ep^2}\, \leq \, {3p-2}
\end{equation}
in $M_T$ and $\eta'(\tau) \, \tau^2 \, \leq \,{9}/{2k} \, \leq \, {5}/{k}$ for every $\tau \in [0, + \infty)$, we finally arrive at
\begin{align}
\mathrm{div}\,Y_{\ep,k}\,\, \geq \,\,\eta_{{k}}\!\left( \!\frac{ (p-1)\vert \nabla u\vert}{\left[ \frac{3-p}{p-1}(1-u) \right]^{\frac{1}{3-p}}} \!\right) \,\mathrm{div}\,Y_\ep 
\,- \, \frac{5}{k} \, \, {c_{p,\ep}^{\frac{p-1}{3-p}} }  \frac{3p-2}{(p-1)^2}\, \max_{M_T} \frac{ |\na\na u| }{{ \left[ \,\frac{3-p}{p-1}\, (1-u) \right]^{\frac{4-p}{3-p}} }} \, .
\end{align}
In particular, the last summand in the right hand side converges to zero uniformly on $M_T$, as $k \to + \infty$. Now, for every $s,t$ as in the statement and every $k \in  \mathbb{N}$, we have that the vector fields $X_{\ep,k}$ coincide with the vector field $X_\ep$ on the boundary of $\{\alpha^\ep_p(s) < u <\alpha^\ep_p(t) \}$, for every $k$ large enough. Hence, the divergence theorem yields
\begin{align}
F^\ep_p(t)& - F^\ep_p(s)  = \!\!\!\int\limits_{\{u=\alpha_{p}(t)\}} \!\!\! \left\langle X_{\ep,k} \,  ,  \frac{\na u}{|\na u|}\right\rangle  d\sigma \,\,- \!\!\!\!\!\!\int\limits_{\{u=\alpha_{p}(s)\}} \!\!\! \left\langle X_{\ep, k} , \frac{\na u}{|\na u|}\right\rangle  d\sigma \,\\
=&\!\!\!\int\limits_{\{\alpha_p(s) < u < \alpha_{p}(t)\}} \!\!\! \!\!\! \mathrm{div}\,X_{\ep, k} \,\, d\mu \, \\
=& \!\!\!\int\limits_{\{\alpha_p(s) < u < \alpha_{p}(t)\}} \!\!\! \left\{\frac{  c_{p,\ep}^{-\frac{(2-p)(p-1)}{(3-p)}}   |\na u |_\ep^{p-2} |\na u|^2 }{\left[ \,\frac{3-p}{p-1}\, (1-u)  \right]^{\frac{p-1}{3-p}+1}} \,
+  \,\mathrm{div}\,Y_{\ep,k} \right\} \,\, d\mu 
\\
\geq& \!\!\!\int\limits_{\{\alpha_p(s) < u < \alpha_{p}(t)\}} \!\!\! \left\{\frac{  c_{p,\ep}^{-\frac{(2-p)(p-1)}{(3-p)}}   |\na u |_\ep^{p-2} |\na u|^2 }{\left[ \,\frac{3-p}{p-1}\, (1-u)  \right]^{\frac{p-1}{3-p}+1}} \,
+ \eta_{{k}}\!\left( \!\frac{ (p-1)\vert \nabla u\vert}{\left[ \,\frac{3-p}{p-1}\, (1-u)  \right]^{\frac{1}{3-p}}} \!\right) \!\mathrm{div}\,Y_{\ep}\right\} \, d\mu 
\\
& \,- \, \frac{5}{k} \, \, |M_T| \, \, {c_{p,\ep}^{\frac{p-1}{3-p}} } \frac{3p-2}{(p-1)^2}\, \, \max_{M_T} \frac{ |\na\na u| }{{ \left[ \,\frac{3-p}{p-1}\, (1-u) \right]^{\frac{4-p}{3-p}} }} \, .
\label{eqcar98}
\end{align}
To proceed, we need to let $k \to + \infty$ in the above inequality. Reasoning as in Subsection~\ref{sub:sard} and using formulas~\eqref{eqff1} and~\eqref{eqff1bis}, we write $\mathrm{div}\,Y_\ep =P_{\ep}+D_{\ep}$, with
\begin{align}
P_\ep&=\frac{c_{p,\ep}^{\frac{p-1}{3-p}} |\na u | }{\left[ \,\frac{3-p}{p-1}\, (1-u) \right]^{\frac{p-1}{3-p} +1}} \biggl[\,\frac{\vert\,\nabla^{\top}\vert \nabla u\vert\vert^{2}}{|\na u|^2}+\vert \ringg{\mathrm{h}}\vert^{2}+\frac{(p-1)(3-p)}{2}\,\,\frac{\big\langle\na |\na u|,\na u\big\rangle^2}{\vert \nabla u\vert_\ep^{4}} \\
&\qquad\qquad\qquad\qquad\qquad\qquad\qquad\qquad\qquad\qquad \, + \, \ep^2\,\frac{\,2 |\na u|^2+\ep^2}{2 |\na u|^4} \,\,\frac{\big\langle\na |\na u|,\na u\big\rangle^2}{\vert \nabla u\vert_\ep^{4}}
\biggr]\\
D_{\ep}&=\frac{c_{p,\ep}^{\frac{p-1}{3-p}} |\na u | }{\left[ \,\frac{3-p}{p-1}\, (1-u) \right]^{\frac{p-1}{3-p} +1}}\biggl[\,\frac{\Ric(\na u,\na u)}{|\na u|^2}\,+\,\frac{(p-1)(5-p)}{(3-p)^{2}}\,\,\frac{\vert \nabla u\vert^{2}}{(1-u)^{2}}\\
&\qquad\qquad\qquad\qquad\quad\qquad\,+\,\frac{(5-p)(p-1)\vert \nabla u\vert^{2}+(p+1)\,\ep^{2}}{(3-p)\vert \nabla u\vert^{2}_{\ep}}\,\,\frac{\bigl\langle\na |\na u|,\na u\bigr\rangle}{(1-u)\,|\na u|}\,\biggr]\,.\label{eqff8}
\end{align}
Noticing that $P_\ep$ is nonnegative and $D_\ep$ is bounded and applying the monotone convergence theorem and the dominated convergence theorem respectively to the sequences $\eta_k P_\ep$ and $\eta_k D_\ep$, one has that the integral of $\eta_k\,\mathrm{div}\,Y_{\ep}$ in formula~\eqref{eqcar98} converges to the integral of $\mathbb{I}_{M_T\setminus\mathrm{Crit}(u)} \,\mathrm{div}\,Y_\ep$, as $k \to + \infty$, where $\mathbb{I}_{M_T\setminus\mathrm{Crit}(u)}$ denotes the characteristic function of $M_T\setminus\mathrm{Crit}(u)$. Hence, passing to the limit and using the coarea formula, one gets
\begin{align*}
F^\ep_p(t) - F^\ep_p(s) & \geq \!\!\!\int\limits_{\{\alpha^\ep_p(s) < u < \alpha^\ep_{p}(t)\}} \!\!\! \left\{\frac{ c_{p,\ep}^{-\frac{(2-p)(p-1)}{(3-p)}} |\na u |_\ep^{p-2} |\na u|^2 }{\left[ \,\frac{3-p}{p-1}\, (1-u) \right]^{\frac{p-1}{3-p}+1}} \,
+ \mathbb{I}_{M_\Tau\setminus\mathrm{Crit}(u)} \,\mathrm{div}\,Y_{\ep} \right\} \,d\mu\\
&= \!\!\!\int\limits_{(\alpha^\ep_p(s) , \alpha^\ep_{p}(t)) \setminus \mathcal{N}_\ep} \!\!\!\!\!\,d\tau \int\limits_{\{ u= \tau \}}\!\!\! 
\left\{\frac{ c_{p,\ep}^{-\frac{(2-p)(p-1)}{(3-p)}} |\na u |_\ep^{p-2} |\na u| }{\left[ \,\frac{3-p}{p-1}\, (1-u) \right]^{\frac{p-1}{3-p}+1}} \,+ \,\frac{\mathrm{div}\,Y_{\ep}}{|\na u|} \right\}d\sigma\\
&= \!\!\!\int\limits_{(\alpha^\ep_p(s) , \alpha^\ep_{p}(t)) \setminus \mathcal{N}_\ep} \!\!\!\!\!\,d\tau \int\limits_{\{ u= \tau \}}\!\!\! \frac{\mathrm{div}\,X_{\ep} }{|\na u|} \,d\sigma\,,
\end{align*}
where $\mathcal{N}_\ep = u(\mathrm{Crit}(u))$ is the set of the critical values of $u$. It is important to notice that in the second passage we benefit from the possibility of using the Sard theorem. Recalling the expression for $\mathrm{div}\,X_{\ep}$ given at the beginning of the proof and using the Gauss--Bonnet theorem, we arrive at 
\begin{align}
F^\ep_p(t) - F^\ep_p(s)& \geq   \!\!\!\int\limits_{(\alpha^\ep_p(s) , \alpha^\ep_{p}(t)) \setminus \mathcal{N}_\ep} \!\!\!\!\!  \frac{c_{p,\ep}^{\frac{p-1}{3-p}}  
\Big( 4\pi - 2\pi \chi(\{ u={\tau} \}) \Big)
}{\left[ \,\frac{3-p}{p-1}\, (1-\tau) \right]^{\frac{p-1}{3-p} +1}} \,d\tau\,\\
&\qquad-\,\ep\bigg(\frac{p+1}{p-1}\bigg)^{\!\!2} \!\!\!\int\limits_{(\alpha^\ep_p(s) , \alpha^\ep_{p}(t)) \setminus \mathcal{N}_\ep} \!\!\!\!\!\!\!\!\!d\tau\,\,  \frac{c_{p,\ep}^{\frac{p-1}{3-p}}}{\left[ \,\frac{3-p}{p-1}\, (1-\tau) \right]^{\frac{p-1}{3-p} +3}} \int\limits_{\{ u= \tau \}}\!\!\!\!\frac{\ep |\na u|^2}{2(p+1)|\na u|^2 + 3 \ep^2} \,d\sigma\,.\label{feq1}
\end{align}
The conclusion then follows by observing that every regular level set of $u$ is connected. The argument follows the very same lines presented at the end of the previous subsection, so we let the details to the reader.
\end{proof}

\subsection{Proof of Theorem~\ref{thm:monotonicity}.} 
\label{sub:proof}
We want to show that, under the assumptions of the statement, the following implication holds true
\begin{equation}
s \leq t \quad \Rightarrow \quad F_p(s) \leq F_p (t) \,,
\end{equation}
for every $s,t \in [t_p, + \infty)$ such that $\alpha_p(s)$ and $\alpha_p(t)$ are regular values of $u_p$. First we notice that if $s=t_p$, then $\{u_p=\alpha_p(s)\} = \{ u_p = 0 \}= \pa M$, which is a regular level set of $u_p$. As $\pa M$ is compact, we have that the tubular neighborhood $\{\alpha_p(s) \leq u_p \leq \alpha_p(t) \}$ is foliated by regular level sets of $u_p$, provided $t-t_p$ is small enough. In this case, the conclusion follows by the direct argument presented in Subsection~\ref{sub:nocrit}.

To prove the remaining cases, we assume then by contradiction that there exist real numbers $t_p<s < t <+\infty$, such that $\alpha_p(s)$ and $\alpha_p(t)$ are regular values for the $p$--capacitary potential $u_p$, but $F_p(t) < F_p(s) $. In particular, we can choose a sufficiently small real number $\delta>0$ such that 
\begin{equation}\label{eq:contr}
0 \,>\,-2\delta \,\geq \,F_p(t) - F_p(s) \,.
\end{equation}
Choosing $T> \alpha_p(t)$, regular value of $u_p$, we have that both the level sets
 $\{u=\alpha_p(s)\}$ and $\{u=\alpha_p(t)\}$ are contained in $M_T= \{0<u_p< T\}$. Now, we recall that on any compact sets sitting inside $M_T \setminus {\rm Crit}(u)$, the solutions $u_p^\ep$ to the approximate problems~\eqref{f2} converge to $u_p$ in the $\mathscr{C}^k$--topology, for every $k \in \mathbb{N}$, as $\ep\to 0$. By means of this property, we are going to prove the following approximation lemma for the function $F_p$, holding at regular values of the corresponding $p$--capacitary potential $u_p$.
\begin{lemma}
\label{lem:Fapp}
Let  $t_p< t< T< + \infty$ be such that $\alpha_p(t)$ is a regular value for the solution $u_p$ of  the problem~\eqref{f1} and let $F_p$ and $F_p^\ep$ be the functions defined in formulas~\eqref{defFp} and~\eqref{defFpep}, respectively. Then, we have
\begin{equation*}
\lim_{\ep \to 0} F_p^\ep(t) \,\,= \,\,F_p(t) \,.
\end{equation*}
\end{lemma}
\begin{proof}
Let $\eta >0$ be such that the tubular neighborhood $\mathscr{U}_\eta = \{ \alpha_p(t-\eta) < u_p< \alpha_p(t+ \eta)\}$ is entirely foliated by regular level sets of $u_p$. Choosing a small enough $\eta>0$, we can assume that also the level sets $\{u_p = \alpha_p(t-\eta)\}$ and $\{u_p = \alpha_p(t+\eta)\}$ are regular. Since $F_p$ is continuous at the regular values of $u_p$, one has that for every $\delta>0$ there exists $\overline\eta=\overline\eta(\delta) >0$ such that $|F_p (t) - F_p(t-\eta)| < \delta/2$, for every $0<\eta<\overline\eta$. 
As the functions $u_p^\ep$ converge uniformly (actually, in the $\mathscr{C}^k$--topology) to the function $u_p$ on $\overline{\mathscr{U}_\eta}$, it is easy to check that there exists $\overline{\ep} = \overline{\ep}(\overline\eta)>0$ such that $\{u_p^\ep=\alpha_p^\ep(t) \} \subseteq \mathscr{U}_\eta$, provided $0< \ep<\overline{\ep}$. In particular, up to choosing $\ep$ small enough, we have that $\{u_p^\ep=\alpha_p^\ep(t) \} \cap \{u_p=\alpha_p(t-\eta) \} = \emptyset$. Moreover, possibly considering a smaller $\overline\ep>0$, it turns out that for every $0<\ep<\overline\ep$, the set $\{u_p^\ep=\alpha_p^\ep(t) \}$ is a regular level set of $u_p^\ep$, as $\na u_p\neq 0$ on $\overline{\mathscr{U}_\eta}$ and $|\na u_p^\ep|$ converge uniformly to $|\na u_p|$ on such a tubular neighborhood. We may then proceed with the following estimate,
\begin{align*}
\left|F_p(t) - F_p^\ep(t)\right| & \,\,\leq \,\,\left|F_p(t) - F_p(t-\eta)\right| \,+ \,\left|F_p(t-\eta) - F_p^\ep(t) \right|\\
& \,\,\leq \,\,\delta/2 \,\,\,+ \Bigg| \,\,\int\limits_{\{u_p=\alpha_{p}(t-\eta)\}} \!\!\! \left\langle X \,, \frac{\na u_p}{|\na u_p|}\right\rangle d\sigma \,\,- \!\!\!\!\!\!\int\limits_{\{u_p^\ep=\alpha_{p}^\ep(t)\}} \!\!\! \left\langle X_{\ep} , \frac{\na u_p^\ep}{|\na u_p^\ep|}\right\rangle d\sigma 	\,\,\Bigg| \,\\
& \,\,\leq \,\,\delta/2 \,\,\,+ \Bigg| \,\,\int\limits_{\{u_p=\alpha_{p}(t-\eta)\}} \!\!\! \left\langle X- X_\ep \,, \frac{\na u_p}{|\na u_p|}\right\rangle d\sigma	\,\,\Bigg|\\
& \,\,\,\quad + \,\Bigg| \,\,\int\limits_{\{u_p=\alpha_{p}(t-\eta)\}} \!\!\! \left\langle X_\ep \,, \frac{\na u_p}{|\na u_p|}\right\rangle d\sigma \,\,- \!\!\!\!\!\!\int\limits_{\{u_p^\ep=\alpha_{p}^\ep(t)\}} \!\!\! \left\langle X_{\ep} , \frac{\na u_p^\ep}{|\na u_p^\ep|}\right\rangle d\sigma 	\,\,\Bigg|\\
& \,\,\leq \,\,\delta/2 \,\,\,+ \Big( \max_{\overline{\mathscr{U}_\eta}}\,\left| X - X_\ep \right| \Big) \,\big| {\{u_p=\alpha_{p}(t-\eta)\}} \big| \,+ \,\Bigg| \,\,\int\limits_{\mathscr{U}_\eta \cap \{ u_p^\ep< \alpha_p^\ep(t)\}} \!\!\!\!\!\!\mathrm{div}\,X_\ep \,d\mu \,\,\Bigg|\\
& \,\,\leq \,\,\delta/2 \,\,\,+ \Big( \max_{\overline{\mathscr{U}_\eta}}\,\left| X - X_\ep \right| \Big) \,\big| {\{u_p=\alpha_{p}(t-\eta)\}} \big|\,+ \,\Big( \max_{\overline{\mathscr{U}_\eta}}\,\left| {\mathrm {div}\,} X_\ep \right| \Big) \,\left| \mathscr{U}_\eta\right| \,,
\end{align*}
where the vector fields $X$ and $X_\ep$ are defined by formulas~\eqref{X_p} and~\eqref{X_pep}, respectively referring to $u_p$ and $u_p^\ep$.
Again, by the smooth convergence of the functions $u_p^\ep$ to $u_p$ on $\overline{\mathscr{U}_\eta}$, we get that $X_\ep$ converge smoothly to $X$ on the same set. In particular, up to choosing $\overline{\eta} = \overline{\eta}(\delta)$ and $\overline\ep= \overline\ep(\overline{\eta}(\delta))$ small enough, the last two summands can be made smaller than $\delta/2$ and we are done.
\end{proof}
We can now conclude the proof of Theorem~\ref{thm:monotonicity}. By the choice of $T$ and Lemma~\ref{lem:Fapp}, there exists $\ep_\delta>0$ such that 
\begin{equation*}
\left| F_p(s) - F_p^\ep(s) \right| \,\leq \,\delta/2 \qquad \hbox{and} \qquad \left| F_p(t) - F_p^\ep(t) \right| \,\leq \,\delta/2 \,,
\end{equation*}
for every $0<\ep< \ep_\delta$. Combining this information together with the assumption~\eqref{eq:contr}, it follows that
\begin{align*} 
- 2 \delta \,\,& \geq \,\,F_p(t) - F_p(s) \,= \,F_p(t) - F_p^\ep(t) +F_p^\ep(t) - F_p^\ep(s) + F_p^\ep(s) - F_p(s)\\
& \geq \,\,F_p^\ep(t) - F_p^\ep(s) - \left| F_p(t) - F_p^\ep(t) \right| - \left| F_p(s) - F_p^\ep(s) \right| \,\,= \,\,F_p^\ep(t) - F_p^\ep(s) -\delta \,.
\end{align*}
Using Lemma~\ref{thm:approx}, one then has that for every $0<\ep< \ep_\delta$, there holds
\begin{align*}
-\delta \,\,& \geq \,\,- \,\,\ep \,\,\Big(\frac{p+1}{p-1}\Big)^{\!2} \!\!\!\! \!\!\!\!\!\!\!\! \int\limits_{\{\alpha_p^\ep(s) < u_p^\ep< \alpha_{p}^\ep(t) \}} \!\!\!\!\!\!\!\!\!\, \frac{\ep |\na u_p^\ep|}{2(p+1)|\na u_p^\ep|^2 + 3 \ep^2} \,\frac{ c_{p,\ep}^{\frac{p-1}{3-p}} \,|\na u_p^\ep|^2}{\left[ \,\frac{3-p}{p-1}\, (1-u_p^\ep) \right]^{\frac{p-1}{3-p} + 3}} 
 \,d\mu\\
& \geq \,\,- \,\,\frac{\ep}{6} \,\,\Big(\frac{p+1}{p-1}\Big)^{\!2} \!\!\!\! \!\!\!\!\!\!\!\! \int\limits_{\{\alpha_p^\ep(s) < u_p^\ep< \alpha_{p}^\ep(t) \}} \!\!\! \frac{ c_{p,\ep}^{\frac{p-1}{3-p}} \,|\na u_p^\ep|^2
}{\left[ \,\frac{3-p}{p-1}\, (1-u_p^\ep) \right]^{\frac{p-1}{3-p} + 3}} 
 \,d\mu 
\end{align*}
where we used the elementary inequality
\begin{equation*}
\frac{\ep |\na u_p^\ep|}{2(p+1)|\na u_p^\ep|^2 + 3 \ep^2} \,\,\leq \,\,\frac{1}{6}\,.
\end{equation*}
Letting $\ep\to0$, we have that the integral in the rightmost hand side of the above inequality converges to 
$$
\int\limits_{\{\alpha_p(s) < u_p< \alpha_{p}(t) \}} \!\!\! \frac{ c_{p}^{\frac{p-1}{3-p}} \,|\na u_p|^2}{\left[ \,\frac{3-p}{p-1}\, (1-u_p) \right]^{\frac{p-1}{3-p} + 3}} 
 \,d\mu\,,
$$
as $u_p^\ep \to u_p$ in the $\mathscr{C}^{1,\beta}$--topology on the compact subsets of $M_T=\{0<u_p< T\}$. This implies that the right hand side is in turn converging to zero, when $\ep \to 0$, leading to the desired contradiction.

\section{Proof of the Riemannian Penrose inequality for a single black hole}\label{sec:RPI}

In light of the monotonicity result obtained in Theorem~\ref{thm:monotonicity}, we present in this section a new proof of the Riemannian Penrose inequality, first proved by Huisken--Ilmanen~\cite{HI} in the case of a single black hole, and by Bray~\cite{bray1} in the general case of multiple black holes. In order to present the precise statement in Theorem~\ref{RPI} below, let us first set up and recall some basic notations and definitions.

\begin{definition}
\label{def:AF} 
A complete $3$--dimensional Riemannian manifold $(M,g)$, with or without boundary, with one single end, is said to be {\em $\mathscr{C}^{1,\alpha}$--asymptotically flat with decay rate $\tau$}, or simply {\em $\mathscr{C}^{1,\alpha}_\tau$--asymptotically flat}, with $\alpha \in (0,1)$ and $\tau>0$, if the following conditions are satisfied:
\begin{enumerate}[label=${\mathrm{(\roman*)}}$]
\item There exists a compact set $K\subseteq M$ such that the end $E=M\setminus K$ is diffeomorphic to the complement of a closed ball in $\R^3$ centered at the origin, through a so--called {\em asymptotically flat coordinate chart} $(E,(x^{1},x^{2} ,x^{3}))$.
\item In such a chart, the metric tensor can be expressed as
\begin{equation*}
g \,= \,g_{ij}\,dx^{i}\otimes dx^{j} \,\,= \,(\delta_{ij} + \gamma_{ij} )\,dx^{i}\otimes dx^{j}
\end{equation*}
with 
\begin{equation}\label{eq4bis}
\sum_{i,\,j}\vert \xx x\xx\vert^{\tau}\,\vert\xx\gamma_{ij}\xx\vert+\sum_{i,\,j, \, k}\vert \xx x\xx\vert^{1+\tau}\,\vert\xx\partial_{k}\gamma_{ij}\xx\vert+\sum\limits_{i,\,j,\,k}\vert \xx x\xx\vert^{1+\tau+\alpha}\,[\partial_{k}\gamma_{ij}]_{\alpha} \,= \,O(1)\,, 
\end{equation}
as $|x|\to + \infty$, where
\begin{equation}\label{holdercondition}
[\partial_{k}\gamma_{ij}]_{\alpha}\,(p)=\!\!\!\!\sup_{q\in E\setminus \{p\}}\frac{\vert\xx \partial_{k}\gamma_{ij}(q)-\partial_{k}\gamma_{ij}(p)\xx \vert}{\vert \xx x(q)-x(p)\xx\vert^{\alpha}}\,,
\end{equation}
for every $p \in E$.
\end{enumerate}
\end{definition}
According to the physicists Arnowitt, Deser and Misner, who first introduced it in~\cite{AdM0}, the ADM mass of an asymptotically flat Riemannian $3$--manifold is defined as the limit
\begin{equation}\label{formADMmass}
m_{\mathrm{ADM}} \,= \,\lim_{r\to +\infty}\,\frac{1}{16\pi} \!\!\! \int\limits_{\{\vert \xx x \xx\vert=r\}}\!\!\!\bigl(\partial_{j}g_{ij}-\partial_{i}g_{jj}\bigr) \,\frac{x^{i}}{\vert \xx x \xx\vert}\,\,d{\sigma}_{\mathrm{eucl}}\,,
\end{equation}
in an asymptotically flat coordinate chart.\\
It can be shown that if the scalar curvature of $(M,g)$ is nonnegative and the decay rate $\tau$ is strictly larger than $1/2$, the ADM mass is a well defined geometric invariant, i.e., the above limit exists (possibly equal to $+\infty$) and its value does not depend on the particular asymptotically flat coordinate chart that is used to compute it (see~\cite{Bartnik} and~\cite{Chrusciel}).

\smallskip

We are now in the position to state and prove the main result of this section.

\begin{theorem}[Riemannian Penrose inequality for a single black hole]
\label{RPI}
Let $(M,g)$ be a $3$--dimensional, complete, connected, noncompact Riemannian manifold with a smooth, compact, connected boundary and one single end. Assume that:
\begin{enumerate}[label=${\mathrm{(\roman*)}}$]
\item The metric $g$ has nonnegative scalar curvature $\RRR\geq 0$.
\item $(M,g)$ is $\mathscr{C}^{1,\alpha}_1$--asymptotically flat, moreover the following Ricci curvature lower bound holds,
\begin{equation}\label{Riccilowerboundmetrica}
\Ric \,\geq \,- \frac{C}{\vert \xx x \xx\vert^{2}}\xx\,g\,, \qquad \hbox{for some $C>0$} \,.
\end{equation}
\item $\pa M$ is the unique closed minimal surface in $(M,g)$. 
\end{enumerate}
Then, the ADM mass satisfies
\begin{equation}
\label{eq:rpi}
m_{\mathrm{ADM}}\,\geq \,\sqrt{	\frac{|\pa M|}{16 \pi}}\,.
\end{equation}
\end{theorem}

\begin{remark}
In~\cite{Chrusciel} Chru\'sciel was able to show that the ADM mass is well defined, under a slightly weaker decay assumption than the one proposed in Definition~\ref{def:AF}. More precisely, it is sufficient to require that 
\begin{equation}
\label{eq:c1tau}
\sum_{i,\,j}\vert \xx x\xx\vert^{\tau}\,\vert\xx\gamma_{ij}\xx\vert+\sum_{i,\,j}\vert \xx x\xx\vert^{1+\tau}\,\vert\xx\partial_{k}\gamma_{ij}\xx\vert \,= \,O(1)\,, \quad \hbox{as $|x| \to + \infty$} \,,
\end{equation}
with $\tau>1/2$ and without any further condition on the decay of the H\"older quotients of the functions $\partial_{k}\gamma_{ij}$. In accordance with the terminology employed in Definition~\ref{def:AF}, it is then natural to refer to condition~\eqref{eq:c1tau} as to {\em $\mathscr{C}^1_\tau$--asymptotical flatness}. On this regard, it is worth pointing out that in~\cite{HI} the Riemannian Penrose inequality was established for the larger class of $\mathscr{C}^1_1$--asymptotically flat manifolds satisfying the Ricci lower bound~\eqref{Riccilowerboundmetrica}. Let us also observe that in~\cite{bray1} the decay conditions
\begin{align}
\sum_{i,\,j,\,k,\,\ell}\vert \xx x\xx\vert^{2+\tau}\,\vert\xx\partial_\ell\partial_k\gamma_{ij}\xx\vert\,&= \,O(1) \,, \quad \hbox{as $|x| \to + \infty$} \\
\vert \xx x\xx\vert^{\zeta}\vert\xx\Ro\xx\vert&=\,O(1) \,, \quad \hbox{as $|x| \to + \infty$} 
\end{align}
for some $\tau>1/2$ and $\zeta>3$, were added to assumption~\eqref{eq:c1tau}. In other words, in order to achieve the general result, Bray asked for a $\mathscr{C}^2_\tau$--asymptotical flatness, together with an extra decay assumption on the scalar curvature. 
\end{remark}

\begin{proof}[Proof of Theorem~\ref{RPI}]
Let us consider the solution $u$ of problem~\eqref{f1}, whose existence and uniqueness are guaranteed by~\cite[Theorem~4.1]{FM} (taking also into account~\cite[Remark 4.2]{FM} and~\cite[Theorem~3.2]{PST14}). 
By virtue of assumption (iii) and the classical facts collected in~\cite[Lemma~4.1]{HI}, we have that $M$ is diffeomorphic to $\R^3 \setminus {B}^3$ and $\partial M$ to $\mathbb{S}^2$. In particular, we have that $H_2(M, \pa M;\mathbb{Z})$ is trivial and we can invoke Theorem~\ref{thm:monotonicity} to deduce that, for every $1<p<3$, the function $F_p$ defined in formula~\eqref{defFp}, is monotone nondecreasing along the regular values of $u$. This implies that
\begin{equation}
\label{eq:fpineq}
F_p(t_p)\,\,\leq\,\,\lim_{\ell\to+\infty}F_p(T_\ell)\,,
\end{equation}
for any sequence $T_\ell\to+\infty$ satisfying the condition that $\alpha_p(T_{\ell})$ is a regular value of $u$, for every $\ell \in \N$.
Using formula~\eqref{eqfp} and recalling that $\{ u = \alpha_p(t_p)\} = \{ u=0 \} = \pa M$ is a minimal surface, one immediately gets
\begin{equation}
{F}_{p}(t_p)\,= \,\,4\pi t_p \,-\,\frac{t_p^{\frac{2}{p-1}}}{c_{p}} \int\limits_{\pa M} \vert \nabla u\vert \,\mathrm{H}\,d\sigma \,+ \,\frac{t_p^{\frac{5-p}{p-1}}}{c_{p}^{2}} \int\limits_{\pa M} \vert \nabla u\vert^{2} \,d\sigma\,\,\geq \,\,4\pi t_p \,.
\end{equation}
A simple algebraic computation based on formulas~\eqref{eqtp} and~\eqref{eqcp}, gives
\begin{equation}\label{eq:limbordo}
(4\pi)^{\frac{2-p}{3-p}} \,\Big(\frac{p-1}{3-p}\Big)^{\!\frac{p-1}{3-p}} \mathrm{Cap}_{p}(\partial M)^{\frac{1}{3-p}}\,= \, 4\pi t_p \,\leq \,F_p(t_p) \,. 
 \end{equation}
We now claim that 
\begin{equation}
\label{eq:limend}
\lim_{\ell\to+\infty} F_{p}(T_\ell) \,\leq \,8\pi m_{\mathrm{ADM}} \,,
\end{equation}
and we defer the proof of such an estimate after the forthcoming Lemmas~\ref{lemM}. Here, we observe that the claim, in combination with inequalities~\eqref{eq:fpineq} and~\eqref{eq:limbordo} leads to
\begin{equation}\label{carlo1}
 \Big(\,\frac{p-1}{3-p}\,\Big)^{\!\frac{p-1}{3-p}}\biggl(\frac{\mathrm{Cap}_{p}(\partial M)}{4\pi}\biggr)^{\!\!\frac{1}{3-p}}
\!\! \leq \,\,2 \,m_{\mathrm{ADM}} \,.
\end{equation}
Using~\cite[Theorem~1.2]{FM} (see also~\cite[Theorem~5.6]{Ago_Fog_Maz_2}) and letting $p \to 1^+$, there holds 
\begin{equation}
\lim_{p \to 1^+} \mathrm{Cap}_{p}(\partial M) \,= \,|\pa M| \,.
\end{equation}
Indeed, assumption (iii) easily implies (see ~\cite[Lemma~4.1]{HI}) that $\pa M$ is also {\em outward minimizing}, meaning that for any smooth domain $E$ containing $\pa M$ it holds $|\pa M| \leq |\pa E|$.
Using this piece of information, the Riemannian Penrose inequality simply follows by letting $p\to 1^+$ in inequality~\eqref{carlo1} {(we underline that this passage to the limit is the reason why we cannot treat the equality case as in~\cite{bray1} and~\cite{HI}).}
However, for the sake of completeness, we provide with the following lemma an explicit lower bound for the $p$--capacity of $\pa M$ in terms of its area $|\pa M|$, that is clearly sufficient for our purposes, once it is combined with estimate~\eqref{xu4}.

\begin{lemma}\label{lempa}
Under the assumptions of Theorem~\ref{RPI}, we have that
\begin{equation}\label{xu4}
\sqrt{{|\pa M|\,}} \,\frac{ \big(\frac{3-p}{2p}\big)^{\!\frac{p}{3-p}} } {C_{\mathrm{Sob}}^{3(p-1)/2(3-p)}} \,\leq \,{\mathrm{Cap}}_p(\partial M)^{\!\frac{1}{3-p}} \,,
\end{equation}
for some positive constant $C_{\mathrm{Sob}}>0$.
\end{lemma}
\begin{proof}
As $(M,g)$ is asymptotically flat, it supports a Sobolev inequality (see~\cite[Theorem~3.2]{PST14} for details), that is, there exists a constant $C_{\mathrm{Sob}}>0$ such that
\begin{equation}
\label{sobolev-inproof}
\Bigl(\int_M {v}^{3/{2}}\,d\mu\Bigr)^{2/3} \leq C_{\mathrm{Sob}} \int_M {\vert\nabla v\vert} \,d\mu
\end{equation}
for any nonnegative $v\in{\mathscr{C}}^{1}_c(M)$. It is well known that, applying this inequality to the function $v^{\frac{2p}{3-p}}$, one gets the $L^p$--Sobolev inequality, for any $1<p<3$
\begin{equation}
\label{psobolev-inproof}
\Bigl(\int_M {v}^{p^*}\,d\mu\Bigr)^{(3-p)/3} \leq \Big(\frac{2p}{3-p}\,C_{\mathrm{Sob}}\Big)^{\!p} \,\int_M {\vert\nabla v\vert}^p \,d\mu \,,\qquad\text{ with }\quad \,\,p^*=\frac{3p}{3-p}\,.
\end{equation}
To proceed, we recall that the $1$--capacity of $\partial M$ can be defined in analogy to the $p$--capacities (see formula~\eqref{pcapacity}) as
\begin{equation}\label{1capacity0}
\mathrm{Cap}_{1}(\partial M) \, = \, \inf \Bigg\{ \int\limits_{M}\vert \nabla w \vert\,d\mu:\, w\in \mathscr{C}^{\infty}_{c}(M), \, w=1\,\, \text{on}\,\, \partial M\Bigg\} \, .
\end{equation}
By truncation and density arguments, an equivalent definition is given by
\begin{equation}
\label{1capacity}
\mathrm{Cap}_{p}(\partial M) \,= \,\inf \Bigg\{\,\int\limits_{M}\vert \nabla w \vert^p\,d\mu:\,w\in \mathscr{C}^1_{c}(M), \,\text{$w$ take values in $[0,1]$,} \text{ $w=1$ on $\partial M$ }\Bigg\}\,,
\end{equation}
holding for every for $1\leq p<3$. Using the latter definition, if $v\in \mathscr{C}^1_{c}(M)$ is a function taking values in $[0,1]$ with $v=1$ on $\partial M$, then $w=v^q$ is a valid competitor for ${\mathrm{Cap}}_1(\partial M)$, provided $q>1$. By the H\"older inequality we get
\begin{equation}\label{xu1}
{\mathrm{Cap}}_1(\partial M) \leq \int\limits_{M} \vert{\nabla v^q}\vert \,d \mu = q \int\limits_{M} v^{q-1} {\vert\nabla v\vert} \,d\mu \leq q \Bigl(\int_{M}v^{{(q-1)}\frac{p}{p-1}} \,d\mu\Bigr)^{{(p-1)}/{p}} \Bigl(\int_{M} {\vert\nabla v\vert}^{p} \,d\mu\Bigr)^{1/p} \, .
\end{equation}
Choosing $q=1+p^*\,\frac{(p-1)}{p}=\frac{2p}{3-p} >1$ in this expression, we can use the above $L^p$--Sobolev inequality to obtain
\begin{align}\label{xu2}
{\mathrm{Cap}}_1(\partial M) &\leq\,\frac{2p}{3-p} \,\Bigl(\int_{M}v^{p^*} \,d\mu\Bigr)^{{(p-1)}/{p}} \Bigl(\int_{M} {\vert\nabla v\vert}^{p} \,d\mu\Bigr)^{1/p}\\
&\leq\,\Big(\frac{2p}{3-p}\Big)^{\!\!{2p}/{(3-p)}} C_{\mathrm{Sob}}^{\,3(p-1)/(3-p)}\Bigl(\int\limits_{M} {\vert\nabla v\vert}^{p} \,d\mu\Bigr)^{\!\!2/(3-p)}.
\end{align}
Taking the infimum in the right hand side of this inequality over the family of functions $v\in \mathscr{C}^1_c(M)$, $v:M\to[0,1]$ and $v=1$ on $\partial M$, we conclude that
\begin{equation}\label{xu3}
{\mathrm{Cap}}_1(\partial M) \,\leq \,\Big(\frac{2p}{3-p}\Big)^{\!\!{2p}/{(3-p)}} C_{\mathrm{Sob}}^{\,3(p-1)/(3-p)} {\mathrm{Cap}}_p(\partial M)^{2/(3-p)}.
\end{equation}
We now observe that in our setting the $1$--capacity of $\pa M$ is always larger than its area, namely
$$
{\mathrm{Cap}}_1(\partial M)\geq \vert{\partial M}\vert \, .
$$ 
Indeed, the coarea formula implies that, for any function $w \in{\mathscr{C}}^\infty_c(M)$ with $w=1$ on $\partial M$, 
\begin{equation}\label{1-ineq-pre}
\int\limits_{M} {\vert\nabla w\vert} \,d\mu \,\geq\,\int_0^1 \vert{\{w = t\}}\vert \,d t \,\geq\,
\inf \big\{ \vert{\partial E}\vert \ :\ {\partial M} \subseteq E, \,\partial E \,\,\text{smooth}\big\} \,\geq\,\vert{\partial M}\vert \,.
\end{equation}
Notice that the last inequality in the above chain holds since $\pa M$ is outward minimizing, as already observed. An alternative argument is by contradiction. If the latter inequality were false, then one could minimize the perimeter among the family of sets that are containing $\pa M$. This would provide a new minimal surface in $M$. Moreover, such a surface would be different from $\pa M$, as is would have a strictly smaller area, against the assumption (iii) in the statement of Theorem~\ref{RPI}. 

Combining the inequality ${\mathrm{Cap}}_1(\partial M)\geq \vert{\partial M}\vert$ with the estimate~\eqref{xu3}, we easily get the desired conclusion.
\end{proof}

We now turn our attention to the proof of claim~\eqref{eq:limend}. As the metric $g$ is $\mathscr{C}^{1, \alpha}_1$--asymptotically flat, with $\alpha \in (0,1)$, it follows from~\cite[Theorem~3.1]{Ben_Fog_Maz_2} that $u$ obeys the asymptotic expansion
\begin{equation}
\label{asy.exp.ofu}
u \,\, = \,\, 1-\frac{p-1}{3-p}\,\frac{c_{p}}{\vert \xx x\xx \vert^{\frac{3-p}{p-1}}}+o_{2}\big(\vert \xx x\xx \vert^{-\frac{3-p}{p-1}}\big)\,.
\end{equation}
In terms of the function $v=u -1+\frac{p-1}{3-p}\,\frac{c_{p}}{\vert \xx x\xx \vert^{\frac{3-p}{p-1}}}$ this corresponds to
\begin{equation}\label{eq4bisbis}
\vert\xx v\xx\vert+\sum_{i}\vert \xx x\xx\vert\,\vert\xx\partial_{i} v\xx\vert+\sum\limits_{i,\,j}\vert \xx x\xx\vert^{2}\,\vert\xx \partial_{ij}^2 v\xx\vert= \,o\big(\vert \xx x\xx\vert^{-\frac{3-p}{p-1}}\big)\,.
\end{equation}
A first consequence of the expansion~\eqref{asy.exp.ofu} is that there exists a real number $\overline{t}\in [t_{p},+\infty)$ such that $\alpha_p(t)$ is a regular value of $u$, for every $t\geq \overline{t}$, so that $\lim_{\ell\to+\infty}F_p(T_\ell)=\lim_{t\to+\infty}F_p(t)$.
The latter limit can be estimated as follows:
\begin{align}
\lim_{t\to+\infty} F_{p}(t)&=\,\,\lim_{t\to +\infty}\frac{4\pi  \,-\,\frac{t^{\frac{3-p}{p-1}}}{c_{p}} \!\!\!\!\int\limits_{\{u=\alpha_{p}(t)\}} \!\!\!\! \vert \nabla u\vert \,\mathrm{H}\,d\sigma \,+ \,\frac{t^{\frac{2(3-p)}{p-1}}}{c_{p}^{2}} \!\!\!\int\limits_{\{u=\alpha_{p}(t)\}} \!\!\!\vert \nabla u\vert^{2} \,d\sigma\,\,}{{1}/{t}}\\
&\leq\,\,\limsup_{t\to+\infty} \frac{{\scalebox{1.3}{$\frac{d}{dt}$}}\,\,\bigg[4\pi  \,-\,\frac{t^{\frac{3-p}{p-1}}}{c_{p}} \!\!\!\!\int\limits_{\{u=\alpha_{p}(t)\}} \!\!\!\! \vert \nabla u\vert \,\mathrm{H}\,d\sigma \,+ \,\frac{t^{\frac{2(3-p)}{p-1}}}{c_{p}^{2}} \!\!\!\int\limits_{\{u=\alpha_{p}(t)\}} \!\!\!\vert \nabla u\vert^{2} \,d\sigma\,\bigg]\,\,}{-{1}/{t^2}}\,, \qquad\quad\label{stima_hop}
\end{align}
where the last inequality follows from the generalized version of de~l'H\^opital's rule in~\cite[Theorem~II]{taylor1}. To apply this result one should actually check that 
$$
\lim_{t \to + \infty} \,\,\,{4\pi  \,-\,\frac{t^{\frac{3-p}{p-1}}}{c_{p}} \!\!\!\!\int\limits_{\{u=\alpha_{p}(t)\}} \!\!\!\! \vert \nabla u\vert \,\mathrm{H}\,d\sigma \,+ \,\frac{t^{\frac{2(3-p)}{p-1}}}{c_{p}^{2}} \!\!\!\int\limits_{\{u=\alpha_{p}(t)\}} \!\!\!\vert \nabla u\vert^{2} \,d\sigma\,\,} = \,\, 0
$$
This condition can be rewritten as
$$
\lim_{t \to + \infty} \,\,  \frac{1}{c_{p}^{p-1}}\!\!\!\!\int\limits_{\{u=\alpha_{p}(t)\}} \!\!\! \bigg[ 1\,-\, \frac{c_{p}^{p-1}\,\vert \nabla u\vert ^{2-p}\,\mathrm{H}}{\,\frac{3-p}{p-1}\, (1-u) \,}\,+\, \frac{c_{p}^{p-1}\,\vert \nabla u\vert ^{3-p}}{\big[\,\frac{3-p}{p-1}\, (1-u) \,\big]^2}\bigg]\,\vert \nabla u\vert^{p-1}\,d\sigma \,\, = \,\, 0
$$
In view of~\eqref{eq24} and~\eqref{eqcp}, it is sufficient to prove that
\begin{equation}\label{eq5}
1\,-\, \frac{c_{p}^{p-1}\,\vert \nabla u\vert ^{2-p}\,\mathrm{H}}{\,\frac{3-p}{p-1}\, (1-u) \,}\,+\, \frac{c_{p}^{p-1}\,\vert \nabla u\vert ^{3-p}}{\big[\,\frac{3-p}{p-1}\, (1-u) \,\big]^2}\, = \, o(1)\,.
\end{equation}
The latter statement follows from the expansions
$$
\vert \nabla u\vert=\frac{c_{p}}{\vert \xx x\xx \vert^{\frac{2}{p-1}}}\,\big(1+o(1)\big)\qquad\text{ and }\qquad\mathrm{H}=\frac{2}{\vert \xx x\xx \vert}\,\big(1+o(1)\big)\,, 
$$
which are straightforward consequences of~\eqref{asy.exp.ofu}. Hence, the estimate~\eqref{stima_hop} is now justified. Computing the derivative at the numerator of~\eqref{stima_hop} (see also~\cite[Section~3]{Ago_Maz_Oro_2}), one can continue the estimate as follows
\begin{align}
\lim_{t\to+\infty} F_{p}(t)\,\, &\leq\,\,\,\limsup_{t\to+\infty}\bigg [-\,\frac{(3-p)}{2(p-1)}\,\,\, t \!\!\!\!\!\!\int\limits_{\{u=\alpha_{p}(t)\}} \!\!\!\!\bigg( \frac{2\vert \nabla u \vert}{\phantom{\Bigl[}\,\frac{3-p}{p-1}\, (1-u) \phantom{\Bigr]}^{\phantom{1}}}\,-\,\mathrm{H}\bigg)^{\!\!2} d\sigma\\
&\phantom{=\,\,\limsup_{t\to+\infty}\Big [}\,\,\,-\,t\!\!\!\!\!\!\int\limits_{\{u=\alpha_{p}(t)\}}\!\!\! \biggl( \frac{\vert\,\nabla^{\Sigma}\vert \nabla u\vert\,\vert^{2}}{\vert \nabla u \vert^{2}}+\frac{\Ro}{2}+\frac{\,\vert \ringg{\mathrm{h}}\vert^{2}}{2} \biggr) \,d\sigma\\
&\phantom{=\,\,\limsup_{t\to+\infty}\Big [}\,\,\,+\,\, t \!\!\!\!\!\!\int\limits_{\{u=\alpha_{p}(t)\}}\!\!\!\!\! \frac{\,\rm{R}^{\Sigma}}{2}\,d\sigma\,\, -\,t\!\!\!\!\!\!\int\limits_{\{u=\alpha_{p}(t)\}}\!\!\! \!\!\!\frac{\,\,\mathrm{H}^2}{4}\,d\sigma\,\,\, \bigg]\\
&\leq\,\,\limsup_{t\to+\infty} \,\,\,\,\,\, t \!\!\!\!\!\!\int\limits_{\{u=\alpha_{p}(t)\}}\!\!\!\!\!\! \frac{\,\rm{R}^{\Sigma}}{2}\,d\sigma\,\, -\,t\!\!\!\!\!\!\int\limits_{\{u=\alpha_{p}(t)\}}\!\!\!\!\!\! \frac{\,\,\mathrm{H}^2}{4}\,d\sigma \\
&=\,\,\limsup_{t\to+\infty} \,\,\,\,\,\frac{t}{4}\,\biggl( \,16\pi\,\,- \!\!\!\!\!\!\int\limits_{\{u=\alpha_{p}(t)\}}\!\!\!\!\!\! \mathrm{H}^{2}\,d\sigma\,\biggr) \,, \label{hopref}
\end{align}
where the last identity follows by the Gauss--Bonnet theorem, as, for large $t$, the level sets are all diffeomorphic to a $2$-dimensional sphere.
Motivated by the above estimate, we now set 
\begin{equation}
{M}_p(t) \,\,= \,\,\frac{t}{4}\,\xx\,\bigg( 16\xx\pi \,\,- \!\!\!\!\!\!\!\int\limits_{\{u=\alpha_{p}(t)\}}\!\!\!\!\!\mathrm{H}^{2}\,d\sigma \bigg)\,,\label{Fp1}
\end{equation}
and we analyze its behavior at infinity with the next lemma.
\begin{lemma}\label{lemM}
Under the assumption of Theorem~\ref{RPI}, we have
\begin{equation}\label{limsupF1p}
\limsup_{t\to+\infty} {M}_p(t) \,\leq \,8\pi m_{\mathrm{ADM}}\,.
\end{equation}
\end{lemma}
\begin{proof}
In the same spirit as in~\cite{HI}, we are going to compare the expression of $M_p$ with an analogous expression in which the geometric quantities are computed with respect to the Euclidean background metric. For this reason, it is convenient to explicitly write the subscript $g$, when a quantity is referred to the original metric. At the same time, we agree that if a quantity is referred to the Euclidean metric, it will be let free of subscripts. The covariant derivative with respect to $g$ will be denoted by $\na$, whereas the symbol $\D$ will indicate the Euclidean covariant derivative. Finally, the level set $\{u = \alpha_p(t)\}$ will be simply denoted by $\Sigma_t$. With these agreements, 
\begin{equation}
{M}_p(t)\,= \,\,\frac{t}{4}\,\xx\,\bigg( 16\xx\pi \,\,- \!\int\limits_{\Sigma_t}\mathrm{H}_g^{2}\,d\sigma_g \bigg)\,.
\end{equation}
Our first task is to obtain an expansion for the $g$--mean curvature $\HHH_g$ of $\Sigma_t$ in terms of its Euclidean mean curvature $\HHH$. In order to do that, we work in an asymptotically flat coordinate chart, like the one of Definition~\ref{def:AF}. As the unit normal vectors to a regular level set $\Sigma_t$ are given by
$$
\nu_g \,= \,\frac{\na u \,\,}{|\na u|_g} \qquad \hbox{and} \qquad \nu \,= \,\frac{\D u}{|\D u|} \,, 
$$
the mean curvatures are readily computed, 
\begin{equation}
\HHH_g \,= \,\left( g^{ij} - \nu_g^i \nu_g^j \right) \,\frac{\na_i \na_j u}{|\na u|_g} \qquad \hbox{and} \qquad \HHH \,= \,\left( \delta^{ij} - \nu^i \nu^j \right)\frac{ \D_i \D_j u}{|\D u|} \,, 
\end{equation}
respectively. The $g$--unit normal is related to the Euclidean one through the formula
\begin{equation}
\nu^i_g \,= \,\Big( 1 + \frac{\gamma(\nu,\nu)}{2}\Big) \,\nu^i - \,\gamma^i_k\,\nu^k \,+\, O(|x|^{-2}) \,,
\end{equation}
where $\gamma^i_k = \delta^{i j} \gamma_{j k}$. Moreover, according to 
formula~\eqref{eq4bis}, there holds
\begin{equation}
g^{ij} \,= \,\delta^{ij}-\gamma^{ij}+\, O(|x|^{-2}) \,\label{exp2} \,,
\end{equation}
where $\gamma^{ij} = \delta^{i \ell} \delta^{k j} \gamma_{k \ell}$, hence, it follows 
\begin{equation}
g^{ij} - \nu_g^i \nu_g^j \,\,= \,\,\delta^{ij} - \nu^i \nu^j \,- \,\big( \delta^{ik} - \nu^i \nu^k \big) \,\gamma_{k \ell} \,\big( \delta^{j \ell} - \nu^j \nu^\ell \big) \, +\, O(|x|^{-2}) \,.
\end{equation}
Noticing that $|\D \D u| / |\D u| = O (|x|^{-1})$, by expansion~\eqref{asy.exp.ofu} and setting $\eta^{i j} = \delta^{ij} - \nu^i \nu^j $, we then arrive at 
\begin{equation}\label{eqHH}
\HHH_g \,= \,\Big( 1 + \frac{\gamma(\nu,\nu)}{2}\Big) \,\HHH - \eta^{ij} \Big( \pa_j g_{ik} - \frac{1}{2} \pa_k g_{ij} \Big) \nu^k - \eta^{ik} \eta^{j \ell} \gamma_{k \ell} \frac{ \D_i \D_j u}{|\D u|} + O(|x|^{-3}) \,,
\end{equation}
which implies 
$$
\HHH^2_g \,= \,\big( 1 + {\gamma(\nu,\nu)} \big) \,\HHH^2 - 2 \HHH \,\eta^{ij} \Big( \pa_j g_{ik} - \frac{1}{2} \pa_k g_{ij} \Big) \nu^k - 2\HHH \,\eta^{ik} \eta^{j \ell} \gamma_{k \ell} \frac{ \D_i \D_j u}{|\D u|} +  O(|x|^{-4}) \,.
$$
As the metric induced on $\Sigma$ by $g$ can be written as $g - \frac{du \otimes du}{|\na u|^2_g}$, the area element can be expressed as 
\begin{equation}
d\sigma_g \,= \,\Big[1+\frac{1}{2}\,\eta^{ij}\gamma_{ij}+ O(\vert \xx x \xx \vert^{-2})\Big] \,d{\sigma}\,. \label{exp4} 
\end{equation}
Putting all together, the Willmore energy integrand then satisfies
\begin{align}
\HHH^2_g \,d\sigma_g\,\,= \,&\,\left[\Big( 1 + \gamma(\nu,\nu) + \frac{\eta^{ij}\gamma_{ij}}{2} \Big) \,\HHH^2 \,\right.\\
&\left. \quad- \,2 \HHH \,\eta^{ij} \Big( \pa_j g_{ik} - \frac{1}{2} \pa_k g_{ij} \Big) \nu^k - \,2\HHH \,\eta^{ik} \eta^{j \ell} \gamma_{k \ell} \frac{ \D_i \D_j u}{|\D u|} + O(|x|^{-4}) \right] \,d{\sigma}.
\quad\qquad \label{eq:will}
\end{align}
Now, we further refine the above expression in light of the asymptotic expansion of the $p$--capacitary potential. Indeed, by means of formula~\eqref{asy.exp.ofu} again, we readily compute
\begin{align}
\frac{ \D_i \D_j u}{|\D u|}& \,= \,\frac{1}{\vertl x\vertr}\Big(\delta_{ij}-\frac{p+1}{p-1}\,\delta_{ik}{\nu}^{k}\,\delta_{j \ell}{\nu^{\ell}}+ o(1)\,\Big)\label{eqf47}\\
{\mathrm{H}}&\,= \,\frac{2}{\vertl x\vertr}\bigl(1+o(1)\bigr)\label{eqf48} 
\end{align}
which implies
\begin{equation}\label{eql49}
2\HHH \,\eta^{ik} \eta^{j \ell} \gamma_{k \ell} \frac{ \D_i \D_j u}{|\D u|} \,= \,\frac{4}{|x|^2} \eta^{k \ell} \gamma_{k \ell} +o(|x|^{-3})\,.
\end{equation}
Plugging this information in formula~\eqref{eq:will}, we obtain
\begin{equation}
\HHH^2_g \,d\sigma_g\,\,= \,\,\left[ \,\HHH^2 \,+ \,\frac{4}{|x|^2} \gamma(\nu , \nu) \,- \,\frac{2}{|x|^2} \eta^{i j} \gamma_{i j} \,- \,\frac{4}{|x|} \eta^{ij} \Big( \pa_j g_{ik} - \frac{1}{2} \pa_k g_{ij} \Big) \nu^k + o(|x|^{-3}) \right] \,d{\sigma} \,.
\end{equation}
To proceed, we now claim that 
\begin{equation}
\label{claimdiv}
\frac{4}{|x|^2} \gamma(\nu , \nu) \,- \,\frac{2}{|x|^2} \eta^{i j} \gamma_{i j} \,= \,\frac{2}{|x|} \bigl( \eta^{ij} \pa_i g_{jk} \nu^k - \mathrm{div}_{\Sigma_{t}} \omega^\top\,\bigr) \,+ \,o(|x|^{-3})\,,
\end{equation}
where $\omega$ is the differential $1$--form defined by $\omega = \gamma_{jk} \nu^k {d} x^j$ and $\omega^\top$ denotes its tangential component. To prove the claim, let us first observe that
$$
\omega = \omega^\top +\omega(\nu)\frac{du}{\vert \D u\vert} = \omega^\top + \gamma(\nu,\nu)\frac{du}{\vert \D u\vert}\qquad\text{ and }\qquad
\pa_i \nu^k \,= \,\eta^{k \ell} \frac{ \D_i \D_\ell u}{|\D u|} \,.
$$
By means of the expansions~\eqref{eqf48} and~\eqref{eql49}, we compute 
\begin{align}
\eta^{ij} \pa_i g_{jk} \nu^k =\eta^{ij} \pa_i \gamma_{jk} \nu^k & = \,\eta^{ij} \pa_i \omega_j \,- \,\eta^{ik} \eta^{j \ell} \gamma_{k \ell} \frac{ \D_i \D_j u}{|\D u|}\\
& = \,\mathrm{div}\,\omega \,- \,\pa_i \omega_j \nu^i \nu^j \,- \,\frac{1}{|x|} \eta^{ij} \gamma_{ij} + o(|x|^{-2})\\
& = \,\mathrm{div}_{\Sigma_t} \omega^\top \,+ \,\gamma(\nu,\nu) \HHH - \,\frac{1}{|x|} \eta^{ij} \gamma_{ij} + o(|x|^{-2})\\
& = \,\mathrm{div}_{\Sigma_t}\omega^\top \,+ \,\frac{2}{|x|}\gamma(\nu,\nu) - \,\frac{1}{|x|} \eta^{ij} \gamma_{ij} +o(|x|^{-2})\,.
\end{align}
Claim~\eqref{claimdiv} then follows with the help of some simple algebra. As a consequence, the expression for the Willmore energy integrand becomes
\begin{align}
\HHH^2_g \,d\sigma_g\,\,= \,\,& \left[ \,\HHH^2 - \frac{2}{|x|} \mathrm{div}_{\Sigma_{t}} \omega^\top + \frac{2}{|x|} \eta^{ij} \pa_i g_{jk} \nu^k - \frac{4}{|x|} \eta^{ij} \pa_j g_{ik}\nu^k + \frac{2}{|x|}\eta^{ij} \pa_k g_{ij}\nu^k+ o(|x|^{-3}) \right] d{\sigma}\\
= \,\,& \left[ \,\HHH^2 - \frac{2}{|x|} \mathrm{div}_{\Sigma_{t}} \omega^\top - \frac{2}{|x|} \delta^{ij} \big( \pa_i g_{jk} - \pa_k g_{ij} \big) \nu^k + o(|x|^{-3}) \right] d{\sigma} \,. \label{eq:nice}
\end{align}
Before integrating this formula, let us observe that, since $u = 1 - (t_p/t)^{(3-p)/(p-1)}$ on $\Sigma_t$, it follows from expansion~\eqref{asy.exp.ofu} that 
\begin{equation}
\frac{1}{|x|} \,= \,\frac{1}{t} + o(|x|^{-1})\qquad \hbox{and} \qquad 
\frac {A}{t^{2}}\,\leq\,\vert \xx \nabla u\xx \vert_g^{p-1} \leq\,\frac{B}{t^{2}} \,\qquad \text{on $\Sigma_t$,}
\end{equation}
for some positive constants $A,B$. Then, integrating $|\na u|_g^{p-1}$ on $\Sigma_t$ with the help of equations~\eqref{eq24} and~\eqref{eqcp}, we get
\begin{equation}
\frac{4 \pi c_p^{p-1}}{B} \,t^2 \,\,\leq \,|\Sigma_t|_g \,\leq \frac{4 \pi c_p^{p-1}}{A} \,t^2 
\end{equation}
and by virtue of equation~\eqref{exp4}, the same estimates hold for the Euclidean area $|\Sigma_t|$ in place of $|\Sigma_t|_g$, up to a different choice of the constants. In view of these considerations, formula~\eqref{eq:nice} becomes
\begin{equation}
\HHH^2_g \,d\sigma_g\,\,= \,\,\left[ \,\HHH^2 - \frac{2}{t} \mathrm{div}_{\Sigma_{t}} \omega^\top - 
\frac{2}{t} \delta^{ij} \big( \pa_i g_{jk} - \pa_k g_{ij} \big) \nu^k + o(t^{-3}) \right] d{\sigma} \,, \quad \text{on $\Sigma_t$\,.}\label{eq:verynice}
\end{equation}
Now we have
\begin{align}
{M}_p(t)& \,= \,\frac{t}{4}\,\xx\,\bigg( 16\xx\pi \,\,- \!\int\limits_{\Sigma_t}\mathrm{H}_g^{2}\,d\sigma_g \bigg)\\
&\,= \,\frac{t}{4}\,\xx\,\bigg( 16\xx\pi \,\,- \!\int\limits_{\Sigma_t}\mathrm{H}^{2}\,d\sigma \bigg) + \,\frac{1}{2} \int\limits_{\Sigma_t} \mathrm{div}_{\Sigma_{t}}\omega^\top\,d\sigma + \,\frac{1}{2} \int\limits_{\Sigma_t} 
\delta^{ij} \big( \pa_i g_{jk} - \pa_k g_{ij} \big) \nu^k \,d\sigma + o(1) \,.
\end{align}
The first summand in the right hand side is nonpositive by the {\em Euclidean Willmore inequality} (see~\cite{Will}), the second summand vanishes by the divergence theorem and it is well known (see~\cite[Proposition~4.1]{Bartnik}) that the third summand tends to $8 \pi m_{\rm ADM}$, as $t \to + \infty$. 
\end{proof}
Combining Lemma~\ref{lemM} with the estimate~\eqref{hopref}, we get
\begin{equation}
\lim_{\ell \to + \infty} F_p(T_\ell) \,=\,\lim_{t \to + \infty} F_p(t) \,\leq \,\limsup_{t \to + \infty} M_p(t) \,\leq \,8 \pi m_{\rm ADM} \,.
\end{equation}
This proves the claim~\eqref{eq:limend} and concludes the proof of the theorem.
\end{proof}

\bibliographystyle{amsplain}
\bibliography{biblio}

\end{document}